\documentclass[12pt,leqno,oneside]{amsart}
\usepackage{mathrsfs,dsfont}
\usepackage{amsmath,amstext,amsthm,amssymb,amscd}
\usepackage{charter}
\usepackage{typearea}
\usepackage{pdfsync}
\usepackage{mathtools}

\usepackage{hyperref}

\usepackage[width=6.7in,height=9in]{geometry}

\numberwithin{equation}{section}

\newtheorem{Theorem}{Theorem}[section]
\newtheorem{Proposition}[Theorem]{Proposition}
\newtheorem{Lemma}[Theorem]{Lemma}

\theoremstyle{definition}
\newtheorem{Definition}[Theorem]{Definition}
\newtheorem{Example}[Theorem]{Example}
\newtheorem{Remark}[Theorem]{Remark}

\newcommand{\wi}{\widetilde}
\newcommand{\wih}{\widehat}

\DeclareMathOperator{\dist}{dist}
\DeclareMathOperator{\Bs}{Bs}

\DeclareMathOperator{\ord}{ord}

\newcommand{\cali}[1]{\mathscr{#1}}
\newcommand{\cO}{\cali{O}} 
\newcommand{\cI}{\cali{I}}

\newcommand{\cC}{\cali{C}}

\newcommand{\field}[1]{\mathbb{#1}}

\newcommand{\R}{\field{R}}
\newcommand{\C}{\field{C}}
\newcommand{\N}{\field{N}}

\newcommand{\Q}{\field{Q}}
\newcommand{\B}{\field{B}}
\renewcommand{\P}{\field{P}}

\newcommand{\X}{\field{X}}

\newcommand{\reg}{\mathrm{reg} }
\newcommand{\sing}{\mathrm{sing} }
\newcommand{\FS}{{\rm FS}}
\newcommand{\eq}{\mathrm{eq} }
\newcommand{\req}{\mathrm{req} }

\newcommand{\PSH}{{\rm PSH}}

\newcommand{\Vol}{{\rm Vol}}

\newcommand{\ddbar}{{\partial\overline\partial}}
\newcommand{\ddc}{{dd^c}}

\newcommand{\comment}[1]{}
\newcommand{\ke}{\nobreak\hspace{.06em plus .03em}}
\newcommand*{\LargerCdot}{\raisebox{-0.25ex}
{\scalebox{1.5}{$\cdot$}}}

\begin{document}

\title[Restricted spaces of holomorphic sections vanishing along subvarieties]
{Restricted spaces of holomorphic sections vanishing along subvarieties}


\author{Dan Coman}
\thanks{D.\ Coman is partially supported by the NSF Grant DMS-2154273 and by the Labex CEMPI (ANR-11-LABX-0007-01)}
\address{Department of Mathematics, Syracuse University, 
Syracuse, NY 13244-1150, USA}\email{dcoman@syr.edu}
\author{George Marinescu}
\address{Universit{\"a}t zu K{\"o}ln, Mathematisches Institut, Weyertal 86-90, 50931 K{\"o}ln, 
Deutschland    \newline
    \mbox{\quad}\,Institute of Mathematics `Simion Stoilow', Romanian Academy, Bucharest, Romania}
\email{gmarines@math.uni-koeln.de}
\thanks{G.\ Marinescu is partially supported 
by the DFG funded projects SFB/TRR 191 
`Symplectic Structures in Geometry, Algebra and Dynamics' 
(Project-ID 281071066-TRR 191), DFG Priority Program 2265 
`Random Geometric Systems' (Project-ID 422743078)
and the ANR-DFG project QuaSiDy (Project-ID 490843120).}
\author{Vi{\^e}t-Anh Nguy{\^e}n}
\address{Universit\'e de Lille, 
Laboratoire de math\'ematiques Paul Painlev\'e, 
CNRS U.M.R. 8524,  \newline
    \mbox{\quad}\,59655 Villeneuve d'Ascq Cedex, 
France
 \newline
    \mbox{\quad}\,Vietnam Institute for Advanced Study in Mathematics (VIASM), 157 Chua Lang Street,
Hanoi, \newline \mbox{\quad}\,Vietnam}
\email{Viet-Anh.Nguyen@univ-lille.fr}
\thanks{V.-A. Nguyen is   supported by 
 the Labex CEMPI (ANR-11-LABX-0007-01), the ANR-DFG project QuaSiDy  
 (ANR-21-CE40-0016), and partially by Vietnam Institute for Advanced Study in Mathematics (VIASM)}

\date{October 7, 2023}

\pagestyle{myheadings}

\begin{abstract}

Let $X$ be a compact normal complex space of dimension $n$ and $L$ be a holomorphic line bundle on $X$. Suppose that $\Sigma=(\Sigma_1,\ldots,\Sigma_\ell)$ is an $\ell$-tuple of distinct irreducible proper analytic subsets of $X$, $\tau=(\tau_1,\ldots,\tau_\ell)$ is an $\ell$-tuple of positive real numbers, and let $H^0_0(X,L^p)$ be the space of holomorphic sections of $L^p:=L^{\otimes p}$ that vanish to order at least $\tau_jp$ along $\Sigma_j$, $1\leq j\leq\ell$. If $Y\subset X$ is an irreducible analytic subset of dimension $m$, we consider the space $H^0_0 (X|Y, L^p)$ of holomorphic sections of $L^p|_Y$ that extend to global holomorphic sections in $H^0_0(X,L^p)$. Assuming that the triplet $(L,\Sigma,\tau)$ is big in the sense that $\dim H^0_0(X,L^p)\sim p^n$, we give a general condition on $Y$ to ensure that $\dim H^0_0(X|Y,L^p)\sim p^m$. When $L$ is endowed with a continuous Hermitian metric, we show that the Fubini-Study currents of the spaces $H^0_0(X|Y,L^p)$ converge to a certain equilibrium current on $Y$. We apply this to the study of the equidistribution of zeros in $Y$ of random holomorphic sections in $H^0_0(X|Y,L^p)$ as $p\to\infty$.
\end{abstract}

\maketitle
\tableofcontents

\section{Introduction} \label{S:Intro}
 
Let $X$ be a compact complex manifold of dimension $n$. If $L$ is a holomorphic line bundle over $X$ we let $L^p := L^{\otimes p}$ and denote by $H^0(X, L^p)$ the space of global holomorphic sections of $L^p$. The line bundle $L$ is called \emph{big} if its Kodaira-Iitaka dimension is equal to the dimension of $X$ (see \cite[Definition 2.2.5]{MM07}). One has that $L$ is big if and only if the volume of $L$ 
\[\Vol_X(L):=\limsup_{p\to\infty} n!\,p^{-n}\dim  H^0 (X, L^p)>0\] 
(see \cite[Theorem 2.2.7]{MM07}). By the Ji-Shiffman/Bonavero/Takayama criterion \cite[Theorem 2.3.30]{MM07}, $L$ is big if and only if it admits a strictly positively curved singular Hermitian metric $h$ (see Section \ref{SS:Preliminaries} for definitions). 

Let $Y\subset X$ be a complex submanifold of dimension $m$. 
To understand ``how many" sections of $L^p|_Y$ 
are restrictions to $Y$ of global sections in $H^0(X,L^p)$, 
Hisamoto considers in \cite{H12} the space 
$H^0(X|Y,L^p):=\{S|_Y:\,S\in H^0(X,L^p)\}$ and the restricted volume 
\[\Vol_{X|Y}(L):=\limsup_{p\to\infty} 
m!\,p^{-m}\dim  H^0 (X|Y, L^p).\]
He studies the asymptotics of the restricted Bergman kernels 
of the spaces $H^0(X|Y,L^p)$ when $L$ is endowed with 
a smooth Hermitian metric $h$, and obtains formulas for 
$\Vol_{X|Y}(L)$ in terms of the Monge-Amp\`ere measure 
related to an equilibrium metric associated to $h$ along $Y$.

Let now $X$ be a compact normal complex space of 
dimension $n$ and $L$ be a holomorphic line bundle over $X$. 
Suppose $\Sigma=(\Sigma_1,\ldots,\Sigma_\ell)$ 
is an $\ell$-tuple of distinct irreducible proper analytic subsets 
of $X$ and $\tau=(\tau_1,\ldots,\tau_\ell)$ is an 
$\ell$-tuple of positive real numbers. 
In \cite{CMN19} we studied the spaces of holomorphic sections 
of $L^p$ that vanish to order at least $\tau_jp$ along
$\Sigma_j$ for all $j=1,\ldots,\ell$. Motivated by \cite{H12}, 
in this paper we consider in addition an analytic subset $Y\subset X$ 
of dimension $m$ and the spaces $H^0_0(X|Y,L^p)$ of sections 
of $L^p|_Y$ which extend to global holomorphic sections of 
$L^p$ on $X$ having the above vanishing properties. 
We study algebraic and analytic objects associated 
to $H^0_0(X|Y,L^p)$, especially the asymptotic growth
of their dimension, and the asymptotics of their Bergman kernels, 
Fubini-Study currents and potentials. We also study the equidistribution 
of zeros of random sequences of sections 
$\{s_p\in H^0_0(X|Y,L^p)\}_{p\geq1}$, as $p\to\infty$.

\par More precisely, in analogy to \cite{CMN19}, 
we consider in this paper the following setting:

\smallskip

(A) $X$ is a compact, irreducible, normal (reduced) 
complex space of dimension $n$, 
$X_\reg$ denotes the set of regular points of $X$, 
and $X_\sing$ denotes the set 
of singular points of $X$.

\smallskip

(B) $L$ is a holomorphic line bundle on $X$.

\smallskip

(C) $\Sigma=(\Sigma_1,\ldots,\Sigma_\ell)$ is an 
$\ell$-tuple of distinct irreducible proper analytic subsets 
of $X$ such that $\Sigma_j\not\subset X_\sing$, 
for every $j\in\{1,\ldots,\ell\}$. We set 
\[\Sigma^\cup=\bigcup_{j=1}^\ell\Sigma_j.\]

\smallskip

(D) $\tau=(\tau_1,\ldots,\tau_\ell)$ is an $\ell$-tuple 
of positive real numbers such that $\tau_j>\tau_k$, 
for every $j,k\in\{1,\ldots,\ell\}$ with $\Sigma_j\subset\Sigma_k$.

\smallskip

(E)  $Y$  is an irreducible proper analytic subset of $X$ 
of dimension $m$  such that 
\[Y\not\subset X_\sing\cup\Sigma^\cup\cup A,\]
where $A=A(L,\Sigma,\tau)$ is the analytic subset 
of $X$ defined in \eqref{e:A}.
\smallskip

For $p\geq1$, let $H^0_0 (X|Y, L^p)$ be the space of sections 
$S\in H^0(Y,L^p|_Y)$ 
which extend to a holomorphic section of $L^p$ on $X$ that vanishes to order at 
least $\tau_{j}p$ along $\Sigma_{j}$, $1\leq j\leq\ell$. Set 
\begin{equation}\label{e:tjp}
t_{j,p}=\begin{cases}\tau_jp  &\text{if $\tau_jp\in\N$} \\
\lfloor\tau_jp\rfloor+1 &\text{if $\tau_jp\not\in\N$}
\end{cases}\;,\;\;1\leq j\leq\ell\,,\,\;p\geq1,
\end{equation}
where $\lfloor r\rfloor$ denotes the greatest integer $\leq r\in\R$. Then 
\begin{equation}\label{e:H00}
H^0_0(X|Y, L^p)=H^0_0(X|Y, L^p,\Sigma,\tau):=\big\{S|_Y:\,S\in H^0_0(X, L^p)\big\}\subset H^0(Y,L^p|_Y),
\end{equation}
where
\begin{equation}\label{e:H00-bis}
H^0_0(X, L^p)=H^0_0(X,L^p,\Sigma,\tau):=\{S\in H^0(X, L^p):\,\ord(S,\Sigma_j)\geq t_{j,p},\;1\leq j\leq\ell\}\,,
\end{equation}
and  $\ord(S,Z)$ denotes the vanishing order of $S$ along 
an irreducible analytic subset $Z$ of $X$, $Z\not\subset X_\sing$. 

To measure the asymptotic growth of the dimension of these spaces, we define the restricted volume of $L$ relative to $Y$ with vanishing along $(\Sigma,\tau)$ by
 \begin{equation}
  \label{e:restricted-vol}
  \Vol_{Y,\Sigma,\tau}(L):=\limsup_{p\to\infty}\frac{\dim H^0_0(X|Y, L^p)}{p^m/m!}\,.
 \end{equation}

Note that, when $Y=X$, $H^0_0(X|X, L^p)=H^0_0(X,L^p)$ are the spaces defined in \eqref{e:H00-bis}, which were introduced and studied in \cite{CMN19}. We recall from \cite{CMN19} the following:
 
\begin{Definition}\label{D:big} 
We say that  the triplet $(L,\Sigma,\tau)$ is {\it big} if 
\[\Vol_{X,\Sigma,\tau}(L)=\limsup\limits_{p\to\infty}\frac{\dim H^0_0 (X, L^p)}{p^n/n!}>0.\]
\end{Definition} 
  
In \cite[Theorem 1.6]{CMN19} we gave a complete characterization of big triplets $(L,\Sigma,\tau)$, in analogy to the Ji-Shiffman/Bonavero/Takayama criterion for big line bundles. We recall this characterization in Section \ref{SS:Div}. Our first main result here is the following: 

\begin{Theorem}\label{T:big}
Let $X,L,\Sigma,\tau$ verify assumptions (A)-(D), and assume that 
$(L,\Sigma,\tau)$ is big and $X$ is a K\"ahler space. 
Then $\Vol_{Y,\Sigma,\tau}(L)>0$ for any analytic subset $Y\subset X$ that verifies (E). 
More precisely, if $Y$ verifies (E) then there exist constants $C>0,p_0\in\N$ such that 
\[\dim H^0_0(X|Y,L^p)\geq Cp^m\,,\,\;\forall\,p>p_0.\]
\end{Theorem}

Theorem \ref{T:big} shows that if $Y$ verifies (E) the dimension of 
the restricted spaces $H^0_0(X|Y, L^p)$ of sections of $L^p$ vanishing along 
$(\Sigma,\tau)$ has the largest possible asymptotic growth $p^{\dim Y}$, 
as soon as the dimension of the ``global" spaces $H^0_0(X,L^p)$ 
grows like $p^{\dim X}$. The proof of Theorem \ref{T:big} is given in Section \ref{S:Dim}.

\medskip
 
Our next result deals with the asymptotics of the Bergman kernel functions and 
Fubini-Study currents associated to the spaces $H^0_0(X|Y,L^p)$. 
Let $X,Y,L,\Sigma,\tau$ verify assumptions (A)-(E), 
and assume in addition that there exists a K\"ahler form $\omega$ 
on $X$ and that $h$ is a singular Hermitian metric on $L$. 
We fix a smooth Hermitian metric $h_0$ on $L$ and write 
\begin{equation}\label{e:varphi}
\alpha:=c_1(L,h_0)\,,\,\;h=h_0e^{-2\varphi}\,,
\end{equation}
 where $\varphi\in L^1(X,\omega^n)$ is called the 
 {\it (global) weight of  $h$  relative to $h_0$}. The metric $h$ is called bounded, 
 continuous, resp.\ H\"older continuous, if $\varphi$ is a bounded, 
 continuous, resp.\ H\"older continuous, function on $X$. 

Let $H^0_{(2)}(X,L^p)=H^0_{(2)}(X,L^p,h^p,\omega^n)$ 
be the Bergman space of $L^2$-holomorphic sections 
of $L^p$ relative to the metric $h^p:=h^{\otimes p}$ and the volume form 
$\omega^n$ on $X$, endowed with the inner product
\[(S,S')_p:=\int_{X}\langle S,S'\rangle_{h^p}\,\frac{\omega^n}{n!}\,,\]
and set $\|S\|_{p}^2:=(S,S)_{p}$. 


We assume in the sequel that the metric $h$ is {\em bounded}. 
Then $h|_Y=h_0|_Ye^{-2\varphi|_Y}$ is a well defined singular metric on $L|_Y$ and we have 
\[H^0_0 (X|Y, L^p)\subset H^0_{(2)}(Y,L^p|_Y,h^p|_Y,\omega^m|_Y).\] 
We use the notation
\[H^0_{0,(2)}(X|Y, L^p)=H^0_{0,(2)}(X|Y, L^p,\Sigma,\tau,h^p,\omega^m):=H^0_0 (X|Y, L^p)\]
when we consider the space $H^0_0 (X|Y, L^p )$ with the inner product induced by $h^p|_Y$ and $\omega^m|_Y$.
 
Let $P^Y_p,\gamma^Y_p$ be the Bergman kernel function and Fubini-Study current of the Bergman space 
$H^0_{0,(2)}(X|Y,L^p)$ defined in \eqref{e:Bkf} and \eqref{e:FS}-\eqref{e:FS_local}. Then  
\begin{equation}\label{e:FSpot}
\frac{1}{p}\,\gamma^Y_p=c_1(L,h)|_Y+\frac{1}{2p}\,dd^c\log P^Y_p=
\alpha|_Y+dd^c\varphi^Y_p\,,\,\text{ where }\,\varphi^Y_p=
\varphi|_Y+\frac{1}{2p}\,\log P^Y_p\,.
\end{equation}
Here $d^c:= \frac{1}{2\pi i}\,(\partial -\overline\partial)$, so $dd^c=\frac{i}{\pi}\,\partial\overline\partial$. We call the function $\varphi^Y_p$ the {\em global Fubini-Study potential} 
of $\gamma^Y_p$. We obtain here the following result on the convergence of the Fubini-Study currents.

\begin{Theorem}\label{T:FSpot}
Let $X,Y,L,\Sigma,\tau$ verify assumptions (A)-(E), and assume that 
$(L,\Sigma,\tau)$ is big and there exists a K\"ahler form $\omega$ on $X$. 
Let $h$ be a continuous Hermitian metric on $L$ and $\alpha,\varphi^Y_p$ be 
defined in \eqref{e:varphi},\ke\eqref{e:FSpot}. Then there exists a weakly  
$\alpha|_Y$-plurisubharmonic function $\varphi^Y_\eq$ on $Y$ such that 
\begin{equation}\label{e:cFSpot}
\varphi^Y_p\to\varphi^Y_\eq\,,\,\;
\frac{1}{p}\,\gamma^Y_p=\alpha|_Y+dd^c\varphi^Y_p\to T^Y_\eq:=
\alpha|_Y+dd^c\varphi^Y_\eq\,, \text{ as } p\to\infty,
\end{equation}
in $L^1(Y,\omega^m|_Y)$, respectively weakly on $Y$.
Moreover, if $h$ is H\"older continuous then there exist a constant 
$C>0$ and $p_0\in\N$ such that 
\begin{equation}\label{e:eFSpot}
\int_Y|\varphi^Y_p-\varphi^Y_\eq|\,\omega^m\leq C\,\frac{\log p}{p}\,,
\,\;\text{for all $p\geq p_0$}\,.
\end{equation}
\end{Theorem}

\begin{Definition}\label{D:TYeq}
The current $T^Y_\eq$ from Theorem \ref{T:FSpot} is called the {\em equilibrium current associated to $(Y,L,h,\Sigma,\tau)$}.
\end{Definition}

Theorem \ref{T:FSpot} is proved in Section \ref{S:FSpot}. The study of the Fubini-Study 
currents associated to various Bergman spaces of holomorphic sections is motivated 
by a foundational result of Tian \cite{Ti90} (see also \cite[Theorem 5.1.4]{MM07}), 
who showed that in the case of a positive line bundle $(L,h)$ on a projective 
manifold $X$ the corresponding Fubini-Study forms $\gamma_p/p\to c_1(L,h)$ as 
$p\to\infty$ in the $\cC^\infty$-topology. This result was generalized in 
\cite[Theorem 5.1]{CM15} to the case of a singular metric $h$ whose curvature 
is a K\"ahler current, by showing that in this case $\gamma_p/p\to c_1(L,h)$ in the 
weak sense of currents. It was further generalized to the case of arbitrary sequences 
of line bundles on compact normal K\"ahler spaces in \cite{CMM17}. The case of 
non-positively curved Hermitian metrics $h$ on big line bundles over projective 
manifolds was treated in \cite{Ber07,Ber09} by considering the equilibrium metric 
associated to $h$, constructed by analogy to extremal plurisubharmonic functions. 
Previously, Bloom \cite{Bl05,Bl09} (cf.\ also Bloom-Levenberg \cite{BL15}) pointed 
out the role of the extremal plurisubharmonic functions in the equidistribution 
theory for random polynomials. More generally, equilibrium metrics with prescribed 
singularities on a line bundle are introduced and studied in \cite{RWN17} 
(see also \cite[Theorem 3]{Dar17}).

We conclude this paper with an application of Theorem \ref{T:FSpot} to the study of the distribution of zeros in $Y$ of random sections in the spaces $H^0_{0,(2)}(X|Y,L^p)$ as $p\to\infty$. To this end, we consider the projective space
\[ \X^Y_p:=\P H^0_{0,(2)}(X|Y,L^p),\; d_p:=\dim H^0_{0,(2)}(X|Y,L^p)-1.\]
We identity $H^0_{0,(2)}(X|Y,L^p)\equiv\C^{d_p+1}$ by using an orthonormal basis, 
and we let $\sigma_p=\omega_\FS^{d_p}$ be the Fubini-Study volume on $\X^Y_p$ induced 
by this identification. Here and in the sequel $\omega_\FS$ denotes the Fubini-Study form 
on a projective space $\P^N$. We also consider the product probability space
\[(\X^Y_\infty,\sigma_\infty):= \prod_{p=1}^\infty (\X^Y_p,\sigma_p)\,.\] 

\begin{Theorem}\label{T:zrhs}
Let $X,Y,L,\Sigma,\tau$ verify assumptions (A)-(E), let $h$ be a singular Hermitian metric on $L$, 
and assume that $(L,\Sigma,\tau)$ is big and there exists a K\"ahler form $\omega$ on $X$.
 
(i) If $h$ is continuous then $\displaystyle\frac{1}{p}\,[s_p=0]\to T^Y_\eq$ 
as $p\to\infty$, in the weak sense of currents on $Y$, 
for $\sigma_\infty$-{\ke}a.{\ke}e.\ $\{s_p\}_{p\geq1}\in\X^Y_\infty$\,.
 
(ii) If $h$ is H\"older continuous then there exists a constant $c>0$ 
with the following property: For any sequence of positive numbers 
$\{\lambda_p\}_{p\geq1}$ such that 
\[\liminf_{p\to\infty} \frac{\lambda_p}{\log p}>(1+m)c\,,\]
there exist subsets $E_p\subset\X^Y_p$ such that, for all $p$ sufficiently large, 

(a) $\sigma_p(E_{p})\leq cp^m\exp(-\lambda_p/c)$\,,

(b) if $s_p\in\X^Y_p\setminus E_p$ we have 
\[\Big|\Big \langle\frac{1}{p}\,[s_p=0]-T^Y_\eq,\phi\Big\rangle\Big|
\leq\frac{c\lambda_p}{p}\,\| \phi\|_{\cC^2}\,,\] 
for any $(m-1,m-1)$-form $\phi$ of class $\cC^2$  on $Y$. 

In particular, the last estimate holds for $\sigma_\infty$-{\ke}a.{\ke}e.\ 
$\{s_p\}_{p\geq1}\in\X^Y_\infty$ provided that $p$ is large enough.
\end{Theorem} 

Theorem \ref{T:zrhs} shows that, as soon as the triplet $(L,\Sigma,\tau)$ is big, the normalized zeros of random holomorphic sections in $H^0_0(X|Y,L^p)$ restricted to suitable analytic subsets $Y\subset X$, distribute as $p\to\infty$ to the equilibrium current $T^Y_\eq$ constructed in Theorem \ref{T:FSpot}. The proof of Theorem \ref{T:zrhs} is given in Section \ref{S:zrhs}. 
 
If $(X,L)$ is a polarized projective manifold Shiffman-Zelditch \cite{SZ99} showed how Tian's theorem
can be applied to obtain the distribution of the zeros of random holomorphic sections of $H^0(X,L^p)$.
Dinh-Sibony \cite{DS06} used meromorphic transforms 
to obtain an estimate on the speed of convergence of the zeros
(see also \cite{DMS} for the non-compact setting).
The result of \cite{SZ99} was generalized to the case of singular metrics
 in \cite{CM15} and further to the case of sequences 
of line bundles over normal complex spaces in \cite{CMM17} (see
also \cite{CM13, DMM16}). In the latter situation, the equidistribution of zeros is considered for general classes of probability measures on the spaces of sections in \cite{BCM,BCHM}. The case of common zeros of random $k$-tuples of sections was considered in \cite{CMN16,CMN18}


 \section{Preliminaries}\label{S:Preliminaries}

We introduce here some notation and we recall a few notions and results that will be used throughout the paper.
 
\subsection{Compact complex manifolds and analytic spaces}\label{SS:Preliminaries}
 
 Let $X$ be a compact complex manifold and $\omega$ 
 be a Hermitian form on $X$. If $T$ is a positive closed current on $X$ 
 we denote by $\nu(T,x)$ the Lelong number of $T$ at $x\in X$ (see e.{\ke}g.\ \cite{D93}). 
 A function $\varphi:\ X\to \R\cup\{-\infty\}$ is called {\it quasi-plurisubharmonic} (quasi-psh) 
 if it is locally the sum of a plurisubharmonic (psh) function and smooth one. 
 Let $\alpha$ be a smooth real closed $(1,1)$-form on $X.$  
 A quasi-psh function $\varphi$ is called {\it $\alpha$-plurisubharmonic} 
 ($\alpha$-psh) if $\alpha+dd^c\varphi\geq 0$ in the sense of currents. 
 We denote by $\PSH(X,\alpha)$ the set of all $\alpha$-psh functions on $X$. 
 The Lelong number of an $\alpha$-psh function $\varphi$ 
 at a point $x\in X$ is defined by $\nu(\varphi,x):=\nu(\alpha+dd^c\varphi,x)$. 

 Since in general the $\ddbar$-lemma does not hold on $X$, 
 we consider the $\ddbar$-cohomology and in particular the 
 space $H^{1,1}_\ddbar(X,\R)$ (see e.{\ke}g.\ \cite{Bo04}). 
 This space is finite dimensional, and if $\alpha$ is a smooth 
 real closed $(1,1)$-form on $X$ we denote its $\ddbar$-cohomology class 
 by $\{\alpha\}_\ddbar$. If $X$ is K\"ahler then $H^{1,1}_\ddbar(X,\R)=H^{1,1}(X,\R)$ 
 and we write $\{\alpha\}_\ddbar=\{\alpha\}$.
 
 \begin{Definition}\label{D:K-currents}
 A positive closed current $T$ of bidegree $(1,1)$ on $X$ is called a 
 {\it  K\"ahler current} if $T\geq \varepsilon\omega$ for some  
 $\varepsilon>0$. A class $\{\alpha\}_\ddbar$ is called {\it big} if it contains a K\"ahler current.  
\end{Definition}

\begin{Definition}\label{D:algsing}
A quasi-psh function $\varphi$ on $X$ is said to have \emph{analytic singularities} if there exists a coherent ideal sheaf $\cI\subset\cO_X$ and $c>0$ such that $\varphi$ can be written locally as 
\begin{equation}\label{e:algsing}
\varphi=\frac{c}{2}\,\log\big(\sum_{j=1}^m|f_j|^2\big)+\psi,
\end{equation}
where $f_1,\ldots,f_m$ are local generators of the ideal sheaf $\cI$ and $\psi$ is a smooth function. 
If $c$ is rational, we furthermore say that $\varphi$ has \emph{algebraic singularities}. Note that $\{\varphi=-\infty\}$ is the support of the subscheme $V(\cI)$ defined by $\cI$.
\end{Definition}

\begin{Definition}\label{D:algsinga}
A quasi-psh function $\varphi$ on $X$ is said to have \emph{almost analytic (resp.\ almost algebraic) singularities} if the following hold: 

$(i)$ $\{\varphi=-\infty\}$ is an analytic subset of $X$,

$(ii)$ $\varphi$ is smooth on $X\setminus\{\varphi=-\infty\}$,

$(iii)$ there exists a proper modification $\sigma:\wi X\to X$, obtained as a finite composition of blow-ups with smooth center and with blow-up locus contained in $\{\varphi=-\infty\}$, such that $\varphi\circ\sigma$ has analytic (resp.\ algebraic) singularities on $\wi X$.
\end{Definition}

If $L$ is a holomorphic line bundle on $X$ and $h^L$ is a singular Hermitian metric on $L$, written $h^L=h^L_0e^{-2\varphi}$ where $h^L_0$ is smooth and $\varphi$ is a quasi-psh function, we say that $h^L$ has (almost) analytic (resp.\ algebraic) singularities if $\varphi$ has (almost) analytic (resp.\ algebraic) singularities. A current $T=\alpha+dd^c\varphi$, where $\alpha$ is a smooth real closed $(1,1)$-form on $X$ and $\varphi$ is a quasi-psh function, is said to have (almost) analytic (resp.\ algebraic) singularities if $\varphi$ has (almost) analytic (resp.\ algebraic) singularities.

Suppose that $\{\alpha\}_\ddbar$ is big. By Demailly's regularization theorem \cite{D92} (see also \cite[Theorem 3.2]{DP04}), one can find a K\"ahler current $T\in\{\alpha\}_\ddbar$ with almost algebraic singularities. The {\em non-K\"ahler locus} of $\{\alpha\}_\ddbar$ is defined in 
\cite[Definition\ 3.16]{Bo04} as the set 
\begin{equation}\label{e:EnK}
E_{nK}\big(\{\alpha\}_\ddbar\big)=
\bigcap\big\{E_+(T):\,\text{$T\in\{\alpha\}_\ddbar$ K\"ahler current}\big\},
\end{equation}
where $E_+(T)=\{x\in X:\,\nu(T,x)>0\}$. Then, by Demailly's regularization theorem \cite{D92}, 
\[E_{nK}\big(\{\alpha\}_\ddbar\big)=
\bigcap\big\{E_+(T):\,\text{$T\in\{\alpha\}_\ddbar$ K\"ahler current with almost algebraic singularities}\big\},\]
hence $E_{nK}\big(\{\alpha\}_\ddbar\big)$ is an analytic subset of $X$. 
It is shown in \cite[Theorem\ 3.17]{Bo04} that there exists a K\"ahler current 
$T\in\{\alpha\}_\ddbar$ with almost algebraic singularities such that 
\[E_+(T)=E_{nK}\big(\{\alpha\}_\ddbar\big).\]

\medskip 

Let now $X$ be a complex space. We write $X=X_\reg \cup X_\sing$, 
where $X_\reg$ and $X_\sing$ are the sets of regular and singular points of $X$. 
We denote by $\PSH(X)$ the set of all psh functions on $X$, 
and by $\PSH(X,\alpha)$ the set of all $\alpha$-psh functions on $X$, 
where $\alpha$ is a smooth real closed $(1,1)$-form on $X$
(see e.g.\ \cite{CMM17,CMN19} for the definitions). 
If $X$ has pure dimension $n$, we consider currents on $X$ as defined in \cite{D85}. 
We denote by $[Z]$ the current of integration along a pure dimensional analytic 
subset $Z\subset X$. If $T$ is a current of bidegree $(1,1)$ on $X$ so that every 
$x\in X$ has a neighborhood $U$ such that $T=dd^c v$ on $U$ for some 
$v \in \PSH(U)$, then $T$ is positive and closed, and we say that $v$ is a 
local potential of $T$. A K\"ahler form on $X$ is a current $T$ as above whose 
local potentials $v$ extend to smooth strictly psh functions in local embeddings 
of $X$ to Euclidean spaces. We call $X$ a K\"ahler space if $X$ admits a 
K\"ahler form (see also \cite[p.\ 346]{Gr62}, \cite [Section 5]{Ohs87}). 
Hermitian forms on $X$ are defined in a similar way by means of local 
embeddings (see e.g.\ \cite{CMM17,CMN19}).

A function $u:X\to[-\infty,+\infty)$ is called {\em weakly psh}, resp.\ {\em weakly $\alpha$-psh}, 
if it is psh, resp.\ $\alpha$-psh, on $X_\reg$ and it is locally upper bounded on $X$. 
If $u$ is weakly psh, resp.\ weakly $\alpha$-psh, then $u$ is locally integrable on $X$ 
and $dd^cu\geq0$, resp.\ $\alpha+dd^cu\geq0$, in the sense of currents on $X$ 
(see \cite[Theorem 1.10]{D85}).  When $X$ is compact, a function $\rho:X\to\R$ 
is called H\"older continuous if, locally, it is H\"older continuous with respect to the metric induced by the Euclidean distance by means of a local embedding of $X$ into $\C^N$. 

If $(L,h)$ is a singular Hermitian holomorphic line bundle over $X$, 
the {\em curvature current} $c_1(L,h)$ of $h$ is defined as in the 
case when $X$ is smooth (see \cite{D90}, \cite{CMM17}). 
We say that $h$ is {\it positively curved}, resp.\ {\it strictly positively curved}, 
if $c_1(L,h)\geq 0$, resp.\ $c_1(L,h)\geq\varepsilon\omega$ for some 
$\varepsilon>0$ and some Hermitian form  $\omega$ on $X$.

\subsection{Bergman kernel functions and Fubini-Study currents}\label{SS:BK-FS}
Let $X$ be as in (A), $Y$ be as in (E), $\omega$ be a Hermitian form and $(L,h)$ be a singular Hermitian holomorphic line bundle on $X$ such that the metric $h$ is {\em bounded}. Since $X$ is compact, the space $H^0(X,L)$ is finite dimensional. The metric $h$ induces by restriction a singular metric $h|_Y$ on $L|_Y$ and we have $c_1(L|_Y,h|_Y)=c_1(L,h)|_Y$.

Let $H^0_{ (2)}(Y, L) = H^0_{ (2)}(Y,L|_Y,h|_Y, \omega^m|_Y)$ be the Bergman space of $L^2$-holomorphic sections of $L|_Y$ relative to the metric $h|_Y$ and the volume form $\omega^m /m!$ on $Y$, endowed with the inner product
\begin{equation}\label{e:inner_product}
(S,S'):=\int_Y\langle S,S'\rangle_h\,\frac{\omega^m}{m!}\,.
\end{equation}

Let $V$ be a subspace of $H^0_{ (2)}(Y, L)$, $r=\dim V$, and $S_1,\ldots,S_r$ be an orthonormal basis of $V$. The {\em Bergman kernel function} $P=P_V$ of $V$ is defined by 
\begin{equation}\label{e:Bkf}
P(x)=\sum_{j=1}^r|S_j(x)|_h^2,\;\;|S_j(x)|_h^2:=\langle S_j(x),S_j(x)\rangle_h,\;x\in Y.
\end{equation}
Note that this definition is independent of the choice of basis. Let $U$ be an open set in $Y$ such that $L$ has a local holomorphic frame $e_U$ on $U$. Then $|e_U|_h=e^{-\varphi_U}$, $S_j=s_je_U$, where $\varphi_U\in L^\infty(U)$, $s_j\in\cO_Y(U)$. It follows that 
\begin{equation}\label{e:Bk_local}
\log P\,|_U= \log\Big(\sum_{j=1}^r|s_j|^2 \Big)-2\varphi_U,
\end{equation}
which shows that $\log P\in L^1(Y,\omega^m|_Y)$.

The Kodaira map determined by $V$ is the meromorphic map given by 
\begin{equation}\label{e:Kodaira_alg_dual}
\Phi=\Phi_V:Y\dashrightarrow \P(V^\star)\,,\,\;\Phi(x)=\{S\in V:\,S(x)=0\},\;x\in Y\setminus \Bs(V),
\end{equation}
where a point in  $\P(V^\star)$ is identified with a hyperplane through the origin in $V$ and $\Bs(V)= \{x \in Y :\,S(x) = 0,\,\forall\,S\in V\}$ is the base locus of $V$. We define the {\em Fubini-Study current}  $\gamma=\gamma_V$ of $V$ by
\begin{equation}\label{e:FS}
\gamma:= \Phi^\star(\omega_\FS),
\end{equation}
where $\omega_\FS$ denotes the Fubini-Study form on $\P(V^\star)$. Then $\gamma$ is a positive closed current of bidegree $(1,1)$ on $Y$, and if $U$ is as above we have
\begin{equation}\label{e:FS_local}
\gamma\mid_U = \frac{1}{2}dd^c \log\Big(\sum_{j=1}^r |s_j|^2\Big).
\end{equation}
Hence by \eqref{e:Bk_local},
\begin{equation}\label{e:Bk_FS}
\gamma= c_1 (L, h)|_Y+ \frac{1}{2}\,dd^c \log P .
\end{equation}

\medskip

Let now $X,Y,L,\Sigma,\tau$ verify assumptions (A)-(E) and $H^0_0 (X|Y, L^p)$ 
be the space defined in \eqref{e:H00}. Since $h$ is a bounded metric on $L$ 
we have $H^0_0 (X|Y, L^p )\subset H^0_{(2)}(Y,L^p|_Y,h^p|_Y,\omega^m|_Y)$.
 The Bergman kernel function $P^Y_p$ and Fubini-Study current 
 $\gamma^Y_p$ of $H^0_0 (X|Y, L^p)$ are called the 
 {\em restricted partial Bergman kernel function}, 
 resp.\ {\em restricted partial Fubini-Study current}, 
 of the space of sections that vanish to order $\tau p$ along $\Sigma$. 
 We have the following variational principle:
\begin{equation}\label{e:var-char}
P^Y_p(x)= \max \left\lbrace |S(x)|^2_{h^p} :\,S\in H^0_0(X|Y, L^p ),\,\|S\|^Y_p 
= 1\right\rbrace,\;x\in Y,
\end{equation}
where $\|\LargerCdot\|^Y_p$ denotes the norm given by the inner product 
in $H^0_{(2)}(Y,L^p|_Y,h^p|_Y,\omega^m|_Y)$.

\subsection{$L^2$-extension theorem for vector bundles}
\label{SS:Hisamoto-ext}
We will need the following variant of the Ohsawa-Takegoshi-Manivel
$L^2$ extension theorem \cite{Man93,OhTa87}
due to Hisamoto \cite[Theorem 1.4]{H12} 
(see also \cite{D00,DKim10,Ohs01})
    
\begin{Theorem}\label{T:Hisamoto-l2-ext} 
Let $X$ be a projective manifold, $Y\subset X$ a complex submanifold, 
$\omega$ a K\"ahler form on $X$, and let $E \to X$  be a holomorphic 
vector bundle with a smooth Hermitian metric $h_E $. 
Then there exist positive constants $N = N (Y, X, h_E ,\omega),\,C =
C(Y, X)$, such that the following holds:

Let $L\to X$ be a holomorphic line bundle with a singular Hermitian metric 
$h_L e^{-2\varphi}$
such that its curvature satisﬁes $c_1(L,h_L)+ \ddc \varphi \geq N \omega$.
Then for any section $s \in H^0 (Y,E \otimes L)$ with
$\int_Y|s|^2 e^{-2\varphi} dV_{\omega,Y} < \infty$,
there exists a section $\widetilde s\in H^0 (X,E \otimes L)$ 
such that $\widetilde s|_Y = s$ and
\[\int_X |\wi s|^2e^{-2\varphi} dV_{\omega,X} \leq  C
\int_Y|s|^2e^{-2\varphi} dV_{\omega,Y},\]
where $|s|$ denotes the norm of $s$ relative to the smooth metric $h_E\otimes h_L$.
\end{Theorem}

  
\section{Dimension of restricted spaces of sections 
vanishing along subvarieties}\label{S:Dim} 

We start by recalling here the characterization of big triplets 
(see Definition \ref{D:big}) by means of divisorizations that was obtained in 
\cite{CMN19}. We then prove Theorem \ref{T:big}.

\subsection{Big triplets and divisorizations}\label{SS:Div}
 
The characterization of big triplets $(L,\Sigma,\tau )$ relies on the following 
consequence of Hironaka's theorem on resolution of singularities: 

\begin{Proposition}{\cite[Proposition 1.4]{CMN19}}\label{P:divisorization} 
Let $X$ and $\Sigma$ verify assumptions (A) and (C). 
Then there exist a compact complex manifold $\wi X$ of dimension $n$ 
and a surjective holomorphic map $\pi:\wi X\to X$, 
given as the composition of finitely many blow-ups with smooth center, 
with the following properties:

(i) There exists an analytic subset $X_\pi\subset X$ such that 
$\dim X_\pi\leq n-2$, $X_\pi\subset X_\sing\cup\Sigma^\cup$, 
$X_\sing\subset X_\pi$, $\Sigma_j\subset X_\pi$ if $\dim\Sigma_j\leq n-2$, 
$E_\pi=\pi^{-1}(X_\pi)$ is a divisor in $\wi X$ that has only normal crossings, 
and $\pi:\wi X\setminus E_\pi\to X\setminus X_\pi$ is a biholomorphism. 

(ii) There exist (connected) smooth complex hypersurfaces 
$\wi\Sigma_1,\ldots,\wi\Sigma_\ell$ in $\wi X$, which have only normal crossings, 
such that $\pi(\wi\Sigma_j)=\Sigma_j$. Moreover, if $\dim\Sigma_j=n-1$ then 
$\wi\Sigma_j$ is the final strict transform of $\Sigma_j$, 
and if $\dim\Sigma_j\leq n-2$ then $\wi\Sigma_j$ 
is an irreducible component of $E_\pi$.

(iii) If $F\to X$ is a holomorphic line bundle and 
$S\in H^0(X,F)$ then $\ord(S,\Sigma_j)=\ord(\pi^\star S,\wi\Sigma_j)$, 
for all $j=1,\ldots,\ell$.
\end{Proposition}
 
\begin{Definition}{\cite[Definition 1.5]{CMN19}}\label{D:divisorization} 
If $\wi X$, $\pi$, $\wi\Sigma:=(\wi\Sigma_1,\ldots,\wi\Sigma_\ell)$, 
verify the conclusions of Proposition \ref{P:divisorization}, 
we say that $(\wi X,\pi,\wi\Sigma)$ is a divisorization of $(X,\Sigma)$.
\end{Definition} 

The following analog of Ji-Shiffman's criterion for big line bundles 
\cite[Theorem 4.6]{JS93} (see
also \cite{Bon93}, \cite[Theorem 2.3.30]{MM07}) was obtained in 
\cite[Theorem 1.6]{CMN19}:

\begin{Theorem}\label{T:CMN3} 
Let $X,L,\Sigma,\tau$ verify assumptions (A)-(D). The following are equivalent: 

(i) $(L,\Sigma,\tau)$ is big;

(ii) For every divisorization $(\wi X,\pi,\wi\Sigma)$ of $(X,\Sigma)$, 
there exists a singular Hermitian metric $h^\star$ on $\pi^\star L$ 
such that $c_1(\pi^\star L,h^\star)-\sum_{j=1}^\ell\tau_j[\wi \Sigma_j]$ 
is a K\"ahler current on $\wi X$ (see Definition \ref{D:K-currents}); 

(iii) There exist a divisorization $(\wi X,\pi,\wi\Sigma)$ of $(X,\Sigma)$ 
and a singular Hermitian metric $h^\star$ on $\pi^\star L$ 
such that $c_1(\pi^\star L,h^\star)-\sum_{j=1}^\ell\tau_j[\wi \Sigma_j]$ 
is a K\"ahler current on $\wi X$;

(iv) There exist $p_0\in\N$ and $c>0$ such that 
$\dim H^0_0(X,L^p)\geq cp^n$ for all $p\geq p_0$.
\end{Theorem}

Given a triplet $(L,\Sigma,\tau)$ and a divisorization $(\wi X,\pi,\wi\Sigma)$ of $(X,\Sigma)$, 
we consider the cohomology class
\begin{equation}\label{e:divcoh1}
\Theta_\pi=\Theta_{\pi,L,\Sigma,\tau}:=c_1(\pi^\star L)-\sum_{j=1}^\ell\tau_j\{\wi\Sigma_j\}_\ddbar\,,
\end{equation}
where $c_1(\pi^\star L)$ is the first Chern class of $\pi^\star L$ 
and $\{\wi\Sigma_j\}_\ddbar\in H^{1,1}_\ddbar(\wi X,\R)$ 
is the class of the current of integration along $\wi\Sigma_j$. 
We have the following simple lemma whose proof is left to the interested reader.

\begin{Lemma}\label{L:divcoh}
In the above setting, the following are equivalent:

(i) There exists a singular Hermitian metric $h^\star$ on $\pi^\star L$ 
such that $c_1(\pi^\star L,h^\star)-\sum_{j=1}^\ell\tau_j[\wi \Sigma_j]$ 
is a K\"ahler current on $\wi X$;

(ii) The class $\Theta_\pi\in H^{1,1}_\ddbar(\wi X,\R)$ is big.
\end{Lemma}

\subsection{Non-K\"ahler loci and blow-ups}\label{SS:nKlocus}
  
We will need certain results regarding the non-K\"ahler locus 
of big cohomology classes. The proofs are included for the convenience of the reader.
    
\begin{Lemma}\label{L:nKpert} Let $(X,\omega)$ 
be a compact Hermitian manifold and 
$\{\alpha\}_\ddbar\in H^{1,1}_\ddbar(X,\R)$ be a big cohomology class. 
Then there exists $\varepsilon>0$ such that if $\eta$ 
is a smooth real closed $(1,1)$-form on $X$ with 
$\eta\geq-\varepsilon\omega$ then $\{\alpha+\eta\}_\ddbar$
is big and $E_{nK}\big(\{\alpha+\eta\}_\ddbar\big)\subset 
E_{nK}\big(\{\alpha\}_\ddbar\big)$.
\end{Lemma}

\begin{proof}
Let $T\in\{\alpha\}_\ddbar$ be a K\"ahler current such that 
$E_+(T)=E_{nK}\big(\{\alpha\}_\ddbar\big)$ and $T\geq\delta\omega$ 
for some $\delta>0$. If $\varepsilon<\delta$ and $\eta\geq-\varepsilon\omega$ 
then $T+\eta\in\{\alpha+\eta\}_\ddbar$ is a K\"ahler current and 
$E_{nK}\big(\{\alpha+\eta\}_\ddbar\big)\subset E_+(T+\eta)=E_+(T)$.
\end{proof}

\begin{Lemma}\label{L:nKblow} Let $(X,\omega)$ 
be a compact Hermitian manifold and $\sigma:\wi X\to X$ 
be a finite composition of blow-ups with smooth center with 
final exceptional divisor $E$. If $\{\alpha\}_\ddbar\in H^{1,1}_\ddbar(X,\R)$ 
is a big cohomology class then 
$\sigma^\star\{\alpha\}_\ddbar\in H^{1,1}_\ddbar(\wi X,\R)$ is big and 
\[E_{nK}\big(\sigma^\star\{\alpha\}_\ddbar\big)\subset
\sigma^{-1}\big(E_{nK}\big(\{\alpha\}_\ddbar\big)\big)\cup E.\]

\end{Lemma}

\begin{proof}
It is well known that there exist $a>0$ and a Hermitian metric $h$ 
on the line bundle $\cO_{\wi X}(E)$ determined by $E$ such that 
$\wi\omega:=\sigma^\star\omega-a\eta$ is a Hermitian form on 
$\wi X$, where $\eta=c_1(\cO_{\wi X}(E),h)$ 
(see e.g.\ \cite[Lemma 2.2]{CMM17}). 
Let $s$ be the canonical section of $\cO_{\wi X}(E)$. 
Then, by the Lelong-Poincar\'e formula, $[E]=\eta+dd^c\log|s|_h$.

We fix a K\"ahler current $T\in\{\alpha\}_\ddbar$ such that 
$E_+(T)=E_{nK}\big(\{\alpha\}_\ddbar\big)$ and $T\geq\delta\omega$ 
for some $\delta>0$. Then $\wi T:=\sigma^\star T-\delta a\eta\geq\delta\wi\omega$ 
is a K\"ahler current on $\wi X$. We have 
\[S:=\sigma^\star T+\delta a\,dd^c\log|s|_h=\wi T+\delta a[E]\geq\delta\wi\omega.\]
So $S\in\sigma^\star\{\alpha\}_\ddbar$ is a K\"ahler current. 
Note that $\sigma:\wi X\setminus E\to X\setminus Z$ is a biholomorphism, 
where $Z$ is a analytic subset of $X$ of codimension $\geq2$ 
such that $E=\sigma^{-1}(Z)$. Since Lelong numbers are biholomorphically 
invariant and the function $\log|s|_h$ is smooth on $\wi X\setminus E$, we infer that
$E_+(S)\subset\sigma^{-1}(E_+(T))\cup E$.
\end{proof}

\subsection{Bonavero's Morse inequalities}\label{SS:Bona}
Bonavero's singular holomorphic Morse inequalities \cite{Bon93} 
have the following consequence which will be needed in the proof of Theorem \ref{T:big}.

\begin{Proposition}\label{P:Bonavero}
Let $(L,h)$ be a singular Hermitian holomorphic line bundle over
 a compact Hermitian manifold $(X,\omega)$ of dimension $n$, 
 such that $h$ has almost algebraic singularities in an analytic subset 
 $A\subset X$ and $c_1(L,h)\geq\varepsilon\omega$ on $X$, 
 where $\varepsilon>0$. If $Y\subset X$ is a (connected) complex 
 submanifold of dimension $m$ such that $Y\not\subset A$, 
 we have as $p\to\infty$ that 
\[\dim H^0_{(2)}(Y,L^p|_Y,h^p|_Y,\omega^m|_Y)\geq
\frac{p^m}{m!}\,\int_{Y\setminus A}c_1(L,h)^m+o(p^m)\geq
\frac{\varepsilon^mp^m}{m!}\,\int_Y\omega^m+o(p^m).\]
\end{Proposition}

\begin{proof}
Note that $h$ defines a singular Hermitian metric $h|_Y$ on $L|_Y$, since $Y\not\subset A$. Let $\sigma:\wi X\to X$ be a proper modification as in Definition \ref{D:algsinga} such that the metric $\wi h:=\sigma^\star h$ on $\wi L:=\sigma^\star L$ has algebraic singularities. Fix a Hermitian form $\wi\omega$ on $\wi X$ such that $\wi\omega\geq\sigma^\star\omega$, and let $\wi Y$ be the strict transform of $Y$ under $\sigma$. Then $\wi Y$ is a complex submanifold of $\wi X$ of dimension $m$. Since $\wi Y\not\subset\sigma^{-1}(A)$ we see that $\wi h$ induces a singular metric $\wi h|_{\wi Y}$ on $\wi L|_{\wi Y}$ which has algebraic singularities. Moreover,
\[H^0_{(2)}(Y,L^p|_Y,h^p|_Y,\omega^m|_Y)\cong 
H^0_{(2)}(\wi Y,\wi L^p|_{\wi Y},\wi h^p|_{\wi Y},\sigma^\star\omega^m|_{\wi Y})\supset 
H^0_{(2)}(\wi Y,\wi L^p|_{\wi Y},\wi h^p|_{\wi Y},\wi\omega^m|_{\wi Y}).\]

If $Z\subset A$ is the blow-up locus of $\sigma$ then $\sigma:\wi X\setminus E\to X\setminus Z$ is a biholomorphism, where $E=\sigma^{-1}(Z)$ is the final exceptional divisor. Thus $\sigma:\wi Y\setminus\sigma^{-1}(A)\to Y\setminus A$ is a biholomorphism. Note that on $\wi Y\setminus\sigma^{-1}(A)$, $\wi h|_{\wi Y}$ is smooth and $c_1\big(\wi L|_{\wi Y},\wi h|_{\wi Y}\big)\geq\sigma^\star\omega>0$. Therefore, by \cite{Bon93} we have as $p\to\infty$ that  
\begin{align*}
\dim H^0_{(2)}(\wi Y,\wi L^p|_{\wi Y},\wi h^p|_{\wi Y},\wi\omega^m|_{\wi Y})&\geq\frac{p^m}{m!}\,
\int_{\wi Y\setminus\sigma^{-1}(A)}c_1\big(\wi L|_{\wi Y},\wi h|_{\wi Y}\big)^m+o(p^m) \\
&=\frac{p^m}{m!}\,\int_{Y\setminus A}c_1(L,h)^m+o(p^m).
\end{align*}
\end{proof}

\subsection{Proof of Theorem \ref{T:big}}\label{SS:Tbig}
Let us start by introducing the analytic subset $A\subset X$ from hypothesis (E). We set
\begin{equation}\label{e:A}
A=A(L,\Sigma,\tau):=\bigcap\big\{\pi\big(E_{nK}\big(\Theta_\pi\big)\big):\,(\wi X,\pi,\wi\Sigma) \text{ is a divisorization of } (X,\Sigma)\big\}.
\end{equation}
Here $\Theta_\pi\in H^{1,1}_\ddbar(\wi X,\R)$ is defined in \eqref{e:divcoh1} and it is a big class by Theorem \ref{T:CMN3}, since the triplet $(L,\Sigma,\tau)$ is big.

Condition (E) implies that we can fix a divisorization $(\wi X,\pi,\wi\Sigma)$ of $(X,\Sigma)$ such that 
\begin{equation}\label{e:spec1}
Y\not\subset X_\sing\cup\Sigma^\cup\cup\pi\big(E_{nK}\big(\Theta_\pi\big)\big).
\end{equation}
We have that $\pi:\wi X\setminus E_\pi\to X\setminus X_\pi$ is a biholomorphism (see Proposition \ref{P:divisorization}). Let $\wi Y$ be the final strict transform of $Y$, and set $\wi L:=\pi^\star L$, $\wi\Sigma^\cup:=\bigcup_{j=1}^\ell\wi\Sigma_j$.

\begin{Lemma}\label{L:embdes1}
There exists a compact complex manifold $\wih X$ of dimension $n$ and a surjective holomorphic map $\wi\pi:\wih X\to\wi X$, given as the composition of finitely many blow-ups with smooth center, such that $\wi\pi:\wih X\setminus\wih E\to\wi X\setminus\wi Z$ is a biholomorphism, where $\wi Z\subset\wi X$ is an analytic subset of dimension $\leq n-2$ and $\wih E=\wi\pi^{-1}(\wi Z)$ is the final exceptional divisor. Moreover, the strict transform $\wih Y$ of $\wi Y$ is a (connected) complex submanifold of $\wih X$ of dimension $m$, and $\wih Y,\wih E_\pi,\wih E$ have simultaneously only normal crossings, where $\wih E_\pi$ denotes the union of the strict transforms under $\wi\pi$ of the irreducible components of $E_\pi$.
\end{Lemma}

\begin{proof}
We apply Hironaka's theorem on the embedded resolution of singularities \cite[Theorems 10.7 and 1.6]{BM97} to $\wi Y\cup E_\pi\subset\wi X$. 
\end{proof}

Set
\[\wih L=\wi\pi^\star\wi L\,,\,\;\wih\Sigma=(\wih\Sigma_1,\ldots,\wih\Sigma_\ell)\,,\,\;\wih\Sigma^\cup:=\bigcup_{j=1}^\ell\wih\Sigma_j,\]
where $\wih\Sigma_j$ is the strict transform of $\wi\Sigma_j$ under $\wi\pi$.
By \cite[Corollary 3.4]{CMN19} we have 
\begin{equation}\label{e:isomH00}
H^0_0(X,L^p,\Sigma,\tau)\cong H^0_0(\wi X,\wi L^p,\wi\Sigma,\tau)\cong H^0_0(\wih X,\wih L^p,\wih\Sigma,\tau), \text{ for all } p\geq1.
\end{equation}

\begin{Lemma}\label{L:iso}
$H^0_0(X|Y,L^p,\Sigma,\tau)\cong H^0_0(\wi X|\wi Y,\wi L^p,\wi\Sigma,\tau)\cong H^0_0(\wih X|\wih Y,\wih L^p,\wih\Sigma,\tau)$, for all $p\geq1$.
\end{Lemma}

\begin{proof} The linear map $\pi^\star:H^0(X,L^p)\to H^0(\wi X,\wi L^p)$, $S\to\pi^\star S$, is bijective with inverse $\pi_\star:H^0(\wi X,\wi L^p)\to H^0(X,L^p)$ defined as follows: if $\wi S\in H^0(\wi X,\wi L^p)$, set $\pi_\star\wi S=S$, 
where $S:=(\pi^{-1})^\star(\wi S|_{\wi X\setminus E_\pi})\in 
H^0(X\setminus X_\pi,L^p|_{X\setminus X_\pi})$ extends to a section in 
$H^0(X,L^p)$ since $X$ is normal and $\dim X_\pi\leq n-2$ \cite[p.\ 143]{GR84}. By \cite[Corollary 3.4]{CMN19}, $\pi^\star:H^0_0(X,L^p)\to H^0_0(\wi X,\wi L^p)$ is an isomorphism. We define a linear map $F:H^0_0(X|Y,L^p)\to H^0_0(\wi X|\wi Y,\wi L^p)$ as follows: if $s\in H^0_0(X|Y,L^p)$ then $s=S|_Y$ for some $S\in H^0_0(X,L^p)$, and we set $F(s)=\pi^\star S|_{\wi Y}$. It is easy to see that $F$ is well defined and bijective. In a similar manner we show that $H^0_0(\wi X|\wi Y,\wi L^p)\cong H^0_0(\wih X|\wih Y,\wih L^p)$.
\end{proof}

\smallskip

\begin{proof}[Proof of Theorem \ref{T:big}]
We use the notation and set-up introduced above, so $\pi(\wi Y)=Y$, $\pi(\wi\Sigma_j)=\Sigma_j$, $\wi\pi(\wih Y)=\wi Y$, $\wi\pi(\wih\Sigma_j)=\wi\Sigma_j$. Let $\wih\omega$ be a Hermitian form on $\wih X$. Since $Y$ verifies \eqref{e:spec1} it follows that $\wi Y\not\subset E_{nK}\big(\Theta_\pi\big)\cup\wi\Sigma^\cup$, hence $\wih Y\not\subset\wi\pi^{-1}\big(E_{nK}\big(\Theta_\pi\big)\big)\cup\wih\Sigma^\cup$. By Lemma \ref{L:nKblow} we have that the class $\wi\pi^\star\Theta_\pi\in  H^{1,1}_\ddbar(\wih X,\R)$ is big and 
\begin{equation}\label{e:spec2}
\wih Y\not\subset E_{nK}\big(\wi\pi^\star\Theta_\pi\big)\cup\wih\Sigma^\cup.
\end{equation}

Let now 
\begin{equation}\label{e:divcoh2}
\wih\Theta:=c_1(\wih L)-\sum_{j=1}^\ell\tau_j\{\wih\Sigma_j\}_\ddbar\in H^{1,1}_\ddbar(\wih X,\R),
\end{equation}
where $c_1(\wih L)$ is the first Chern class of $\wih L$ and 
$\{\wih\Sigma_j\}_\ddbar\in H^{1,1}_\ddbar(\wih X,\R)$ is 
the class of the current of integration along $\wih\Sigma_j$.
Using \eqref{e:divcoh1} we infer that $\wih\Theta=\wi\pi^\star\Theta_\pi+\{R\}_\ddbar$, 
for some positive closed current $R$ of bidegree $(1,1)$ supported in $\wih E$. 
This implies that $\wih\Theta$ is a big class and 
\begin{equation}\label{e:wihTheta}
E_{nK}\big(\wih\Theta\big)\subset 
E_{nK}\big(\wi\pi^\star\Theta_\pi\big)\cup\wih E.
\end{equation} 
Indeed, if $T\in\wi\pi^\star\Theta_\pi$ is a K\"ahler current with 
$E_+(T)=E_{nK}\big(\wi\pi^\star\Theta_\pi\big)$ then 
$T+R\in\wih\Theta$ is a K\"ahler current and 
$E_+(T+R)\subset E_+(T)\cup\wih E$. By using \eqref{e:spec2} it follows that 
\begin{equation}\label{e:spec3}
\wih Y\not\subset E_{nK}\big(\wih\Theta\big)\cup\wih\Sigma^\cup.
\end{equation}

Consider the class
\begin{equation}\label{e:divcoh3}
\wih\Theta_r:=c_1(\wih L)-\sum_{j=1}^\ell r_j\{\wih\Sigma_j\}_\ddbar=\wih\Theta+\sum_{j=1}^\ell(\tau_j-r_j)\{\wih\Sigma_j\}_\ddbar,\;r_j\in\Q,\;r_j>\tau_j.
\end{equation}
By Lemma \ref{L:nKpert} we have that $\wih\Theta_r$ is big and $E_{nK}\big(\wih\Theta_r\big)\subset E_{nK}\big(\wih\Theta\big)$ if $r_j-\tau_j$ is small enough. Hence by \eqref{e:spec3},  
\begin{equation}\label{e:spec4}
\wih Y\not\subset E_{nK}\big(\wih\Theta_r\big)\cup\wih\Sigma^\cup.
\end{equation}

By Demailly's regularization theorem \cite{D92} (see also \cite[Theorem 3.2]{DP04}) 
and by \cite[Theorem 3.17]{Bo04} there exists a K\"ahler current $\wih T\in\wih\Theta_r$ 
with almost algebraic singularities such that $E_+(\wih T)=E_{nK}\big(\wih\Theta_r\big)$. 
Then $\wih T+\sum_{j=1}^\ell r_j[\wih\Sigma_j]\in c_1(\wih L)$, so there exists a 
singular metric $\wih h$ on $\wih L$ such that 
$c_1(\wih L,\wih h)=\wih T+\sum_{j=1}^\ell r_j[\wih\Sigma_j]$. Since $\wih T$ 
has almost algebraic singularities and $r_j>0$ are rational, it follows easily 
that the metric $\wih h$ has almost algebraic singularities contained in 
$E_{nK}\big(\wih\Theta_r\big)\cup\wih\Sigma^\cup$. Moreover 
we have $c_1(\wih L,\wih h)\geq\wih T\geq\varepsilon\wih\omega$, 
for some $\varepsilon>0$. Thanks to \eqref{e:spec4} we can apply 
Proposition \ref{P:Bonavero} to obtain, as $p\to\infty$,
\begin{equation}\label{e:Bon}
\dim H^0_{(2)}(\wih Y,\wih L^p|_{\wih Y},
\wih h^p|_{\wih Y},\wih\omega^m|_{\wih Y})\geq
\frac{\varepsilon^mp^m}{m!}\,\int_{\wih Y}\wih\omega^m+o(p^m).
\end{equation}

Since $X$ is a K\"ahler space it follows that $\wih X$ is a K\"ahler 
manifold (see e.g.\ \cite[Lemma 2.2]{CMM17}). Moreover, 
\eqref{e:isomH00} implies that the line bundle $\wih L$ is big, 
as $(L,\Sigma,\tau)$ is a big triplet. Hence $\wih X$ is a projective 
manifold. By Theorem \ref{T:Hisamoto-l2-ext}, if $p$ is sufficiently large, every section 
$s\in H^0_{(2)}(\wih Y,\wih L^p|_{\wih Y},\wih h^p|_{\wih Y},\wih\omega^m|_{\wih Y})$ 
extends to a section $S\in H^0_{(2)}(\wih X,\wih L^p,\wih h^p,\wih\omega^n)$. 
Since $c_1(\wih L,\wih h)\geq r_j[\wih\Sigma_j]$ the metric $\wih h$ has a global 
quasi-psh weight with Lelong number $\geq r_j$ along $\wih\Sigma_j$. Thus 
$S$ must vanish to order $\lfloor r_jp\rfloor$ on $\wih\Sigma_j$, $1\leq j\leq\ell$.
As $r_j>\tau_j$ we have $\lfloor r_jp\rfloor>\tau_jp$ for all $p$ sufficiently large,
so $S\in H^0_0(\wih X,\wih L^p,\wih\Sigma,\tau)$. It follows that 
$H^0_{(2)}(\wih Y,\wih L^p|_{\wih Y},\wih h^p|_{\wih Y},\wih\omega^m|_{\wih Y})
\subset H^0_0(\wih X|\wih Y,\wih L^p,\wih\Sigma,\tau)$, and the proof of Theorem 
\ref{T:big} is concluded by \eqref{e:Bon} and Lemma \ref{L:iso}.
\end{proof}

\begin{Remark}\label{R:Vol}
In the setting of Theorem \ref{T:big} assume in addition that
$Y\cap\Sigma^{\cup}=\emptyset$. Since the triplet
$(L,\Sigma,\tau)$ is big it follows that $L$ is big.
By \cite{H12} we have that $\Vol_{X|Y}(L)>0$ when $X,Y$ are smooth.
Theorem \ref{T:big} actually shows that $\Vol_{Y,\Sigma,\tau}(L)>0$,
i.\,e.\ the dimension of the space of sections of $L^p|_Y$ that
extend to $X$ and vanish at least to order $p\tau$ on $\Sigma$,
grows like $p^{\dim Y}$.

\end{Remark}
An important special situation is the one when $X$ is smooth and $\Sigma_j$ are 
analytic hypersurfaces. We recall the characterization of big triplets in this case, 
which follows from \cite[Theorem 1.3]{CMN19} and Lemma \ref{L:divcoh}.

\begin{Theorem}\label{T:CMN3spec}
Let $X,L,\Sigma,\tau$ verify (A)-(D) and assume that $X$ is smooth 
and $\dim\Sigma_j=n-1$, $1\leq j\leq\ell$. The following are equivalent: 

(i) $(L,\Sigma,\tau)$ is big;

(ii) The class $\Theta=\Theta_{L,\Sigma,\tau}:=c_1(L)-
\sum_{j=1}^\ell\tau_j\{\Sigma_j\}_\ddbar\in H^{1,1}_\ddbar(X,\R)$ is big.
\end{Theorem}

As above, $\{\Sigma_j\}_\ddbar$ is the class of the current of integration 
$[\Sigma_j]$. In this case the exceptional set $A$ from \eqref{e:A} can 
be described more precisely:

\begin{Proposition}\label{P:A}
In the setting of Theorem \ref{T:CMN3spec}, if $(L,\Sigma,\tau)$ is big then 
$A\subset E_{nK}(\Theta)\cup\Sigma^\cup$. Hence Theorem \ref{T:big} holds 
for any $Y$ that verifies the assumption 

\medskip

\noindent
{\rm (E*)} $Y$  is an irreducible proper analytic subset of $X$ 
of dimension $m$  such that $Y\not\subset E_{nK}(\Theta)\cup\Sigma^\cup$.
\end{Proposition}

\begin{proof}
Let $(\wi X,\pi,\wi\Sigma)$ be a divisorization of $(X,\Sigma)$. We apply 
\eqref{e:wihTheta} to $\pi:\wi X\to X$, $\Theta$ and the class 
$\Theta_\pi$ from \eqref{e:divcoh1}. Using Lemma \ref{L:nKblow} we infer that 
\[E_{nK}\big(\Theta_\pi\big)\subset E_{nK}\big(\pi^\star\Theta\big)\cup 
E_\pi\subset\pi^{-1}\big(E_{nK}(\Theta)\big)\cup E_\pi.\]
By Proposition \ref{P:divisorization}, $\pi(E_\pi)=X_\pi\subset\Sigma^\cup$, hence 
$\pi\big(E_{nK}\big(\Theta_\pi\big)\big)\subset E_{nK}(\Theta)\cup\Sigma^\cup$.
\end{proof}

We conclude this section with a simple example that illustrates Theorem \ref{T:big}. 

\begin{Example}\label{E:Pn}
Let $X=\P^n$, $\tau_j>0$, and $\Sigma_j$ be irreducible analytic hypersurfaces 
in $X$ of degree $d_j$, where $1\leq j\leq\ell$. Let $L=\mathcal O(d)$, 
where $d>\sum_{j=1}^\ell\tau_jd_j$. Then $c_1(L)-\sum_{j=1}^\ell\tau_j\{\Sigma_j\}$ 
is a K\"ahler class, so the triplet $(L,\Sigma,\tau)$ is big by Theorem \ref{T:CMN3spec}. 
Moreover, $H^0_0(X,L^p)$ is given by the space of homogeneous polynomials of 
degree $dp$ in $\C[z_0,\ldots,z_n]$ which are divisible by $\prod_{j=1}^\ell P_j^{t_{j,p}}$, 
where $P_j$ is an irreducible polynomial of degree $d_j$
such that $\Sigma_j=\{P_j=0\}$. Let now $Y\subset X$ be an irreducible analytic subset of 
dimension $m$ such that $Y\not\subset\Sigma^\cup$, i.e.\ $Y$ verifies (E*). By 
Theorem \ref{T:big} we have that the space $H^0_0(X|Y,L^p)$ of restrictions of  
polynomials in $H^0_0(X,L^p)$ to $Y$ verifies $\dim H^0_0(X|Y,L^p)\geq Cp^m$ 
for all $p$ sufficiently large.
\end{Example}
 
\section{Convergence of the Fubini-Study potentials}\label{S:FSpot}

In this section we introduce a certain restricted extremal quasi-psh function 
with poles along a divisor, which will be used to define the equilibrium potential 
and current from Theorem \ref{T:FSpot}. We refer to \cite{LS99,RS05} for similar 
constructions in the case of psh Green functions with poles 
along analytic sets. 
In the absence of the poles our envelope coincides with the 
restricted equilibrium 
weight introduced by Hisamoto  \cite[Definition 3.1]{H12}. 
We then proceed 
with the proof of Theorem \ref{T:FSpot}.

\subsection{Envelopes of quasi-psh functions with 
poles along a divisor}\label{SS:env}
Let $(X,\omega)$ be a compact Hermitian manifold 
of dimension $n$, $Y\subset X$ be a complex submanifold 
of dimension $m$, $\Sigma_j\subset X$ be irreducible 
complex hypersurfaces, and 
let $\tau_j>0$, where $1\leq j\leq\ell$. Assume that 
\[Y\not\subset\Sigma^\cup:=\bigcup_{j=1}^\ell\Sigma_j.\]
We write $\Sigma=(\Sigma_1,\ldots,\Sigma_\ell)$, 
$\tau=(\tau_1,\ldots,\tau_\ell)$, 
and let $\dist$ be the distance on $X$ induced by $\omega$.

Let $\alpha$ be a smooth real closed $(1,1)$-form on $X$. 
We fix a smooth Hermitian metric $g_j$ on $\cO_X(\Sigma_j)$, 
let $s_{\Sigma_j}$ be the canonical section of 
$\cO_X(\Sigma_j)$, $1\leq j\leq\ell$, and set 
\begin{equation}\label{e:bts}
\beta_j=c_1(\cO_X(\Sigma_j),g_j)\,,\,\;\theta=
\alpha-\sum_{j=1}^\ell\tau_j\beta_j\,,\,\;\sigma_j:=
|s_{\Sigma_j}|_{g_j}.
\end{equation}
As in \cite[(4.2)]{CMN19} we consider the class
\begin{equation}\label{e:calL}
\mathcal{L}(X, \alpha,\Sigma,\tau)=\{\psi\in\PSH(X,\alpha):\,
\nu(\psi,x)\geq\tau_j,\,\forall\,x\in\Sigma_j,\,1\leq j\leq\ell\}\,
\end{equation}
of $\alpha$-psh functions with logarithmic poles of order $\tau_j$ 
along $\Sigma_j$. Given a function $\varphi:Y\to\R\cup\{-\infty\}$ 
we introduce the following subclasses of quasi-psh 
functions and their upper envelopes: 
\begin{align}
\mathcal{A}(X|Y,\alpha,\Sigma,\tau,\varphi)&=
\{ \psi\in\mathcal{L}(X,\alpha,\Sigma,\tau):\,
\psi\leq\varphi\text{ on }Y\}, \label{e:calA} \\
\mathcal{A}'(X|Y,\alpha,\Sigma,\tau,\varphi)&=
\Big\{ \psi'\in\PSH(X,\theta):\,\psi'\leq\varphi-
\sum_{j=1}^\ell\tau_j\log\sigma_j\text{ on }
Y\setminus\Sigma^\cup\Big\},  \label{e:calA'} \\
\varphi^Y_\eq(x)=\varphi^Y_{\eq,\Sigma,\tau}(x)&=
\sup\{\psi(x):\,\psi\in\mathcal{A}(X|Y,\alpha,\Sigma,\tau,\varphi)\},
\;x\in Y,\label{e:enveq} \\
\varphi^Y_\req(x)=\varphi^Y_{\req,\Sigma,\tau}(x)&=
\sup\{\psi'(x):\,\psi'\in\mathcal{A}'(X|Y,\alpha,\Sigma,\tau,\varphi)\},
\;x\in Y.  \label{e:envreq}
\end{align}

We call $\varphi^Y_\eq$ the {\em equilibrium envelope of 
$(\alpha,Y,\Sigma,\tau,\varphi)$}, and $\varphi^Y_\req$ 
the {\em reduced equilibrium envelope of 
$(\alpha,Y,\Sigma,\tau,\varphi)$}. 
Note that when $Y=X$ these coincide with the equilibrium envelopes 
$\varphi_\eq,\varphi_\req$ defined in \cite[Section 4]{CMN19}. 
However, when $Y\neq X$ it is possible that 
$\varphi_\eq=\varphi_\req=+\infty$ on $X\setminus Y$. 
The following result is concerned with some basic properties 
of these envelopes. Its proof is very similar to that of 
\cite[Proposition 4.1]{CMN19}, so we omit it.

\begin{Proposition}\label{P:envelopes}
Let $X,Y,\Sigma,\tau,\alpha,\theta$ be as above, and let 
$\varphi:\ Y\to\R\cup\{-\infty\}$ be an upper semicontinuous function. 
Then the following hold: 

(i) The mapping 
$\PSH(X,\theta)\ni\psi'\mapsto\psi:=\psi'+
\sum_{j=1}^\ell\tau_j\log\sigma_j\in\mathcal{L}(X,\alpha,\Sigma,\tau)$ 
is well defined and bijective, with inverse 
$\mathcal{L}(X,\alpha,\Sigma,\tau)\ni\psi\mapsto\psi':=
\psi-\sum_{j=1}^\ell\tau_j\log\sigma_j\in\PSH(X,\theta)$.

(ii) There exists a constant $C>0$ depending only on 
$X,Y,\Sigma,\tau,\alpha,\theta$ such that $\sup_Y\psi'\leq\sup_Y\varphi+C$, 
for every $\psi'\in\mathcal{A}'(X|Y,\alpha,\Sigma,\tau,\varphi)$. 

(iii) $\mathcal{A}(X|Y,\alpha,\Sigma,\tau,\varphi)\neq\emptyset$ 
if and only if $\mathcal{A}'(X|Y,\alpha,\Sigma,\tau,\varphi)\neq\emptyset$. 
Moreover, in this case we have that $(\varphi^Y_\eq)^\star\in PSH(Y,\alpha|_Y)$ and 
$(\varphi^Y_\req)^\star\in PSH(Y,\theta|_Y)$, where the upper semicontinuous regularization is taken along $Y$, and
\begin{equation}\label{e:eqreq}
\varphi^Y_\eq=\varphi^Y_\req+\sum_{j=1}^\ell\tau_j\log\sigma_j\,\text{ on $Y$.}
\end{equation}

(iv) If $\varphi$ is bounded and there exists a bounded $\theta$-psh function on $X$, then $\varphi^Y_\req$ is bounded. 

(v) If $\PSH(X,\theta)\neq\emptyset$ and $\varphi_1,\varphi_2:\ Y\to\R$ 
are bounded and upper semicontinuous, then 
\[\varphi^Y_{1,\req}-\sup_Y|\varphi_1-\varphi_2|\leq\varphi^Y_{2,\req}\leq
\varphi^Y_{1,\req}+\sup_Y|\varphi_1-\varphi_2|\]
holds on $Y$. Moreover, if $\varphi_1\leq\varphi_2$ then 
$\varphi^Y_{1,\req}\leq\varphi^Y_{2,\req}$.
\end{Proposition}

It is worth noting that we may obtain a regularity of $\varphi^Y_\eq$ in terms of $\varphi$  using  the technique developed in \cite{DMN17} and \cite{CMN19}. 

\subsection{Proof of Theorem \ref{T:FSpot}}\label{SS:FSpot}
 
Let $X,Y, L,\Sigma,\tau$ verify assumptions (A)-(E), 
and assume in addition that there exists a K\"ahler form 
$\omega$ on $X$ and that $h$ is a continuous Hermitian metric 
on $L$. Let $h_0,\varphi$ be as in \eqref{e:varphi}. 
Let $P^Y_p,\gamma^Y_p$ be the Bergman kernel function 
and Fubini-Study current of the space $H^0_{0,(2)}(X|Y,L^p)$, 
and let $\varphi^Y_p$ be the global Fubini-Study potential 
of $\gamma^Y_p$ (see \eqref{e:FSpot}). 

We use the set-up and notation introduced in Section \ref{SS:Tbig}. 
Namely, $(\wi X,\pi,\wi\Sigma)$ is a divisorization of $(X,\Sigma)$ 
such that \eqref{e:spec1} holds, and $\wi\pi:\wih X\to\wi X$ 
is a resolution of singularities as in Lemma \ref{L:embdes1}. 

\begin{Lemma}\label{L:embdes2}
Let $\wih\pi:=\pi\circ\wi\pi:\wih X\to X$ and 
$Z:=X_\pi\cup\pi(\wi Z)$. Then $Z\subset X$ 
is an analytic subset of dimension $\leq n-2$, 
$\wih\pi^{-1}(Z)=\wih E_\pi\cup\wih E$, 
$\wih\pi:\wih X\setminus(\wih E_\pi\cup\wih E)\to 
X\setminus Z$ is a biholomorphism, and 
$\wih\pi^\star\omega>0$ on $\wih X\setminus(\wih E_\pi\cup\wih E)$.
\end{Lemma}

\begin{proof}
Note that $\pi(\wi Z)\subset X$ is an analytic subset of dimension $\leq n-2$, 
by Remmert's proper mapping theorem. We have that $\pi^{-1}(Z)=E_\pi\cup\wi Z$, 
so $\wih\pi^{-1}(Z)=\wi\pi^{-1}(E_\pi\cup\wi Z)=\wih E_\pi\cup\wih E$, 
and the lemma follows.
\end{proof}

Let $\wih\omega$ be a K\"ahler form on $\wih X$ such that 
$\wih\omega\geq\wih\pi^\star\omega$ (see e.{\ke}g.\ \cite[Lemma 2.2]{CMM17}) 
and denote by $\dist$ the distance on $\wih X$ induced by $\wih\omega$. Set 
\begin{equation}\label{e:wivarphi}
\wih L:=\wih\pi^\star L\,,\,\;\wih h_0:=\wih\pi^\star h_0\,,\,\;\wih\alpha:=
\wih\pi^\star\alpha=c_1(\wih L,\wih h_0)\,,\,\;\wih\varphi:=
\varphi\circ\wih\pi\,,\,\;\wih h:=\wih\pi^\star h=\wih h_0e^{-2\wih\varphi}\,.
\end{equation}
We write $\wih h^p=\wih h^{\otimes p}$ and $\wih h_0^p=
\wih h_0^{\otimes p}$. Lemma \ref{L:iso} 
implies that the map 
\begin{equation}\label{e:iso}
S\in H^0_{0,(2)}(X|Y,L^p)\to\wih\pi^\star S\in H^0_{0,(2)}(\wih X|\wih Y,\wih L^p)=
H^0_{0,(2)}(\wih X|\wih Y,\wih L^p,\wih\Sigma,\tau,\wih h^p,\wih\pi^\star\omega^m)
\end{equation}
is an isometry. It follows that 
\begin{equation}\label{e:wiFS}
\wih P^{\wih Y}_p=P^{ Y}_p\circ\wih\pi\,,\,\;\wih\gamma^{\wih Y}_p=\wih\pi^\star\gamma^{ Y}_p
\end{equation}
are the Bergman kernel function, resp.\ Fubini-Study current, 
of the space $H^0_{0,(2)}(\wih X|\wih Y,\wih L^p)$. Note that $\wih\pi(\wih Y)=Y$ and 
\begin{equation}\label{e:wiFSpot}
\frac{1}{p}\,\wih\gamma^{\wih Y}_p=\wih\alpha|_{\wih Y}+dd^c\wih\varphi^{\wih Y}_p,
\text{ where }\,\wih\varphi_p^{\wih Y}=\wih\varphi|_{\wih Y}+\frac{1}{2p}\,\log\wih P^{\wih Y}_p=
\varphi_p^Y\circ\wih\pi.
\end{equation}

Let $\wih\varphi^{\wih Y}_\eq$ be the equilibrium envelope of 
$(\wih\alpha,\wih Y,\wih\Sigma,\tau,\wih\varphi)$ defined in \eqref{e:enveq},
\begin{equation}\label{e:wienveq}
\wih\varphi^{\wih Y}_\eq(x)=
\sup\{\psi(x):\,\psi\in\mathcal{L}(\wih X,\wih\alpha,\wih\Sigma,\tau),\;
\psi\leq\wih\varphi \,\text{ on } \wih Y\},\;x\in\wih Y,
\end{equation}
where $\mathcal{L}(\wih X,\wih\alpha,\wih\Sigma,\tau)$ is defined 
in \eqref{e:calL}. Let $s_{\wih\Sigma_j}$ be the canonical section of 
$\cO_{\wih X}(\wih\Sigma_j)$ and fix a smooth Hermitian metric $g_j$ on 
$\cO_{\wih X}(\wih\Sigma_j)$ such that 
\begin{equation}\label{e:wisig}
\sigma_j:=|s_{\wih\Sigma_j}|_{g_j}<1 \text{ on } \wih X,\;1\leq j\leq\ell.
\end{equation}
Set 
\begin{equation}\label{e:wibts}
\beta_j=c_1(\cO_{\wih X}(\wih\Sigma_j),g_j)\,,\,\;\wih\theta=
\wih\alpha-\sum_{j=1}^\ell\tau_j\beta_j\,.
\end{equation}
Note that $[\wih\Sigma_j]=\beta_j+dd^c\log\sigma_j$, 
by the Lelong-Poincar\'e formula. Moreover 
$\{\wih\theta\}=\wih\Theta$, where $\wih\Theta$ 
is the big class defined in \eqref{e:divcoh2}. In this setting, 
we first prove the convergence of the global Fubini-Study potentials on $\wih Y$.


\begin{Theorem}\label{T:wiFSpot}
Let $X,Y,L,\Sigma,\tau$ verify assumptions (A)-(E), 
and assume that $(L,\Sigma,\tau)$ is big and there exists a K\"ahler form 
$\omega$ on $X$. Let $h$ be a continuous Hermitian metric on $L$, 
let $\varphi$, $\wih\varphi$, $\wih\varphi^{\wih Y}_p,\wih\varphi^{\wih Y}_\eq,\wih\theta$ 
be defined in \eqref{e:varphi}, \eqref{e:wivarphi}, \eqref{e:wiFSpot}, 
\eqref{e:wienveq}, resp.\ \eqref{e:wibts}, and set 
$\wih Z:=E_{nK}\big(\{\wih\theta\}\big)\cup\wih\Sigma^\cup$. 
Then the following hold:
\\[2pt]
(i) $\wih\varphi^{\wih Y}_p\to\big(\wih\varphi^{\wih Y}_\eq\big)^\star$ 
in $L^1(\wih Y,\wih\omega^m|_{\wih Y})$ 
and locally uniformly on $\wih Y\setminus\wih Z$ as $p\to\infty$.
\\[2pt]
(ii) If $\varphi$ is H\"older continuous on $Y$ 
then there exist a constant $C>0$ and $p_0\in\N$ such that 
for all $y\in\wih Y\setminus\wih Z$ and $p\geq p_0$ we have 
\[\big|\wih\varphi^{\wih Y}_p(y)-\big(\wih\varphi^{\wih Y}_\eq\big)^\star(y)\big|\leq
\frac{C}{p}\,\big(\log p+\big|\log\dist(y,\wih Z)\big|\big).\]
\end{Theorem}
The proof is done by estimating the partial Bergman kernel 
$\wih P^{\wih Y}_p$ 
from \eqref{e:wiFS} as in \cite[Theorem 5.1]{CMN19} 
(see also \cite{Ber07}, \cite{Ber09}, \cite{CM15}, \cite{RWN17} 
for similar approaches). Let 
\[\Omega_{\wih\varphi}(\delta)=
\sup\big\{|\wih\varphi(x)-\wih\varphi(y)|:\,x,y\in\wih Y,\;\dist(x,y)<\delta\big\}\]
be the modulus of continuity of $\wih\varphi$. 

\begin{Proposition}\label{P:uewiFS}
In the setting of Theorem \ref{T:wiFSpot}, there exists a constant $C>0$ 
such that for all $p\geq1$ and $\delta\in(0,1)$ the following estimate holds on $\wih Y$:
\[\wih\varphi^{\wih Y}_p\leq\wih\varphi^{\wih Y}_\eq+C\Big(\delta+\frac{1}{p}-
\frac{\log\delta}{p}\Big)+2\Omega_{\wih\varphi}(C\delta)\,.\]
\end{Proposition}

\begin{proof} We proceed along the same lines as in the proof 
of \cite[Proposition 5.4]{CMN19}, working with $\wih Y$ instead 
of $\wih X$. Using Lemma \ref{L:embdes2} and following 
the proof of \cite[Proposition 5.2]{CMN19} we can show that 
there exists a constant $C>0$ such that 
\begin{equation}\label{e:Fp}
F_p(\delta):=\sup\Big\{\frac{1}{2p}\,\log\wih P^{\wih Y}_p(y):\,y\in\wih Y,\;
\dist(y,\wih E_\pi\cup\wih E)\geq\delta\Big\}\leq
\frac{C}{p}\,(1-\log\delta)+\delta+\Omega_{\wih\varphi}(\delta)\,,
\end{equation}
for $p\geq1$ and $0<\delta<1$.

By Lemma \ref{L:embdes1}, $\wih Y,\wih E_\pi,\wih E$ have simultaneously 
only normal crossings. So $\wih E_\pi\cap\wih Y,\wih E\cap\wih Y$ are 
divisors in $\wih Y$ that have simultaneously only normal crossings. 
Therefore the argument from the proof of \cite[Proposition 5.4]{CMN19} goes 
through without changes and shows that there exists a constant $C'>0$ such 
that for all $y\in\wih Y$, $p\geq1$ and $0<\delta<1$ we have 
\begin{equation}\label{e:wivarphi1}
\wih\varphi^{\wih Y}_p(y)\leq\wih\varphi(y)+C'\delta+\Omega_{\wih\varphi}(C'\delta)+F_p(\delta/C').
\end{equation}

Recall that the sections in $H^0_{0,(2)}(\wih X|\wih Y,\wih L^p)$ are restrictions
 to $\wih Y$ of sections in $H^0_0(\wih X,\wih L^p,\wih\Sigma,\tau)$. Therefore we infer 
 from \eqref{e:wiFSpot} that the function $\wih\varphi^{\wih Y}_p$ is the restriction 
 to $\wih Y$ of an $\wih\alpha$-psh function $v$ on $\wih X$ with Lelong number 
 $\geq t_{j,p}/p\geq\tau_j$ along $\wih\Sigma_j$, $1\leq j\leq\ell$. So 
 $v\in\mathcal{L}(\wih X,\wih\alpha,\wih\Sigma,\tau)$, and by \eqref{e:wivarphi1} 
 and \eqref{e:wienveq}
\[\wih\varphi^{\wih Y}_p\leq\wih\varphi^{\wih Y}_\eq+C'\delta+
\Omega_{\wih\varphi}(C'\delta)+F_p(\delta/C').\]
The proof is concluded by applying \eqref{e:Fp}.
\end{proof}

We now obtain a lower bound on $\log\wih P^{\wih Y}_p$ 
and $\wih\varphi^{\wih Y}_p$. Recall that $\{\wih\theta\}=\wih\Theta$ 
is a big class, where $\wih\theta$ is defined in \eqref{e:wibts} and $\wih\Theta$ 
in \eqref{e:divcoh2}.

\begin{Lemma}\label{L:eta}
There exists a $\wih\theta$-psh function $\eta$ with almost algebraic singularities on $\wih X$ such that  
\begin{equation}\label{e:eta1}
\{\eta=-\infty\}=E_{nK}\big(\{\wih\theta\}\big)\,,\,\;\eta\leq
-1\,,\,\; \wih\theta+dd^c\eta\geq\varepsilon_0\wih\omega
\geq\varepsilon_0\wih\pi^\star\omega
\end{equation}
hold on $\wih X$, for some constant $\varepsilon_0>0$. 
Moreover, there exist constants $N_0,M_0>0$ such that 
\begin{equation}\label{e:eta2}
\eta(x)\geq-N_0\big|\log\dist\big(x,E_{nK}
\big(\{\wih\theta\}\big)\big)\big|-M_0,\;x\in\wih X.
\end{equation}
\end{Lemma}

\begin{proof} The existence of $\eta$ satisfying 
\eqref{e:eta1} follows directly from \cite[Theorem 3.17]{Bo04} 
and Demailly's regularization theorem \cite{D92,DP04}. 
Moreover, by \cite[Proposition 3.7]{D92} $\eta$ has locally 
the form $\eta=c\log\big(\sum_{j=1}^\infty|f_j|^2\big)+\psi$,
where $c>0$ is rational, $f_j$ are holomorphic functions, 
and $\psi$ is a bounded function. Since the ring of germs 
of holomorphic functions is Noetherian there exists $k$ 
such that locally $E_{nK}\big(\{\wih\theta\}\big)=
\{f_1=\ldots=f_k=0\}$. Thus 
$\eta\geq c\log\big(\sum_{j=1}^k|f_j|^2\big)+c'$, 
and \eqref{e:eta2} follows from the \L ojasiewicz inequality. 
\end{proof}

\begin{Proposition}\label{P:lewiFS}
In the setting of Theorem \ref{T:wiFSpot}, 
there exist a constant $C>0$ and $p_0\in\N$ such that for all 
$p\geq p_0$ the following estimate holds on $\wih Y\setminus\wih Z$:
\[\wih\varphi^{\wih Y}_p\geq\big(\wih\varphi^{\wih Y}_\eq\big)^\star+\frac{C}{p}\,\eta+
\frac{1}{p}\,\sum_{j=1}^\ell\log\sigma_j>-\infty\,.\]
\end{Proposition}

\begin{proof} Using Choquet's lemma, we can find 
an increasing sequence of functions $\{\psi_k\}_{k\geq1}
\subset\mathcal{A}(\wih X|\wih Y,\wih\alpha,\wih\Sigma,\tau,\wih\varphi)$ 
such that $\psi_k\nearrow\big(\wih\varphi^{\wih Y}_\eq\big)^\star$ 
a.e.\ on $\wih Y$.
Let 
\begin{equation}\label{e:rho}
\rho:=\eta+\sum_{j=1}^\ell\tau_j\log\sigma_j
\in\mathcal{L}(\wih X,\wih\alpha,\wih\Sigma,\tau)
\cap\cC^\infty(\wih X\setminus\wih Z),
\end{equation}
where $\eta$ is the function from Lemma \ref{L:eta}. Then 
\begin{equation}\label{e:ddcrho}
\wih\alpha+dd^c\rho=\wih\theta+dd^c\eta+
\sum_{j=1}^\ell\tau_j[\wih\Sigma_j]\geq\varepsilon_0\wih\omega.
\end{equation}
Since $\wih\varphi$ is bounded there exists $a\in\R$ 
such that $\rho\leq\wih\varphi+a$ on $\wih X$.
Replacing $\psi_k$ by $\max\{\psi_k,\rho-a\}$ we obtain a sequence 
\begin{equation}\label{e:psik}
\psi_k\in\mathcal{A}(\wih X|\wih Y,\wih\alpha,\wih\Sigma,\tau,
\wih\varphi),\;\psi_k\geq\rho-a \text{ on } \wih X,\;
\psi_k\nearrow\big(\wih\varphi^{\wih Y}_\eq\big)^\star
\text{ a.e.\ on }\wih Y.
\end{equation}

Consider the Bergman space $H^0_{(2)}(\wih X,\wih L^p,H_{p,k},
\wih\omega^n)$, where the metric $H_{p,k}$ on $\wih L^p$ is given by
\begin{equation}\label{e:psipk}
H_{p,k}:=\wih h_0^pe^{-2\psi_{p,k}}\,,\,\;\psi_{p,k}
=(p-p_0)\psi_k+p_0\rho+\sum_{j=1}^\ell\log\sigma_j,
\end{equation}
and $p_0\in\N$ will be specified later. We have that 
$\psi_{p,k}\in L^\infty_{loc}(\wih X\setminus\wih Z)$,
and $\psi_{p,k}\leq p\wih\varphi+p_0a$ on $\wih Y$ 
since $\sigma_j<1$. Moreover, by \eqref{e:ddcrho} 
and since $\wih\alpha+dd^c\psi_k\geq0$ we obtain
\[c_1(\wih L^p,H_{p,k})=(p-p_0)(\wih\alpha+dd^c\psi_k)+
p_0(\wih\alpha+dd^c\rho)+\sum_{j=1}^\ell([\wih\Sigma_j]-\beta_j)
\geq(p_0\varepsilon_0-C_1)\wih\omega\]
for every $k\geq1$, where $C_1>0$ is a constant such that 
$\sum_{j=1}^\ell\beta_j\leq C_1\wih\omega$. 
By \eqref{e:spec3}, the singular metric $H_{p,k}|_{\wih Y}$
on $\wih L|_{\wih Y}$ is well defined and 
$c_1(\wih L^p|_{\wih Y},H_{p,k}|_{\wih Y})\geq
(p_0\varepsilon_0-C_1)\wih\omega|_{\wih Y}$. 
Therefore, if $p_0$ is chosen large enough we can apply 
the $L^2$\ke-{\ke}estimates for $\overline\partial$ from 
\cite{D82} (see also \cite[Theorem 5.5]{CMN19}) 
and proceed as in the proofs of \cite[Theorem 5.1]{CM15} 
and \cite[Proposition 5.6]{CMN19}, working on $\wih Y$, 
to show the following:
there exist $C_2>0$ and $p_0\in\N$ such that for all $k\geq1$, 
$p\geq p_0$ and $y\in\wih Y\setminus\wih Z$ there exists 
$S_{y,p,k}\in H^0_{(2)}(\wih Y,\wih L^p|_{\wih Y},
H_{p,k}|_{\wih Y},\wih\omega^m|_{\wih Y})$ with 
$S_{y,p,k}(y)\neq 0$ and 
\begin{equation}\label{e:peak1}
0<\int_{\wih Y}|S_{y,p,k}|^2_{H_{p,k}|_{\wih Y}}\,
\frac{\wih\omega^m}{m!}\leq 
C_2|S_{y,p,k}(y)|^2_{H_{p,k}|_{\wih Y}}<+\infty.
\end{equation}

Note that $\wih X$ is projective since it is K\"ahler and 
$\wih L$ is a big line bundle. Applying Theorem 
\ref{T:Hisamoto-l2-ext} and increasing $p_0$ 
if necessary (so that $p_0\varepsilon_0-C_1>N$), 
we infer that $S_{y,p,k}$ extends to a section 
$\wih S_{y,p,k}\in H^0_{(2)}(\wih X,\wih L^p,H_{p,k},
\wih\omega^n)$. Using \eqref{e:rho},\ke\eqref{e:psik},
\ke\eqref{e:psipk}, we see that the quasi-psh function 
$\psi_{p,k}$ has Lelong number $\geq p\tau_j+1$ 
along $\wih\Sigma_j$, $1\leq j\leq\ell$. 
Hence $H^0_{(2)}(\wih X,\wih L^p,H_{p,k},
\wih\omega^n)\subset H^0_0(\wih X,\wih L^p,\wih\Sigma,\tau)$
and $S_{y,p,k}\in H^0_0(\wih X|\wih Y,\wih L^p,\wih\Sigma,\tau)$. 
By \eqref{e:wivarphi},\ke\eqref{e:psipk} we get
\[H_{p,k}=\wih h_0^pe^{-2\psi_{p,k}}=
\wih h^pe^{2p\wih\varphi-2\psi_{p,k}}, 
\text{ so } H_{p,k}\geq\wih h^pe^{-2p_0a} \text{ on } \wih Y.\]
As $\wih\omega\geq\wih\pi^\star\omega$ 
we obtain by \eqref{e:peak1}
\[e^{-2p_0a}\int_{\wih Y}|S_{y,p,k}|^2_{\wih h^p|_{\wih Y}}\,
\frac{(\wih\pi^\star\omega)^m}{m!}\leq 
C_2|S_{y,p,k}(y)|^2_{\wih h^p|_{\wih Y}}
e^{2p\wih\varphi(y)-2\psi_{p,k}(y)}.\]
Using \eqref{e:var-char} this yields that
\begin{equation}\label{e:peak2}
\wih P^{\wih Y}_p(y)\geq C_2^{-1}e^{2\psi_{p,k}(y)-
2p\wih\varphi(y)-2p_0a},\;\forall\,k\geq1,p\geq p_0,y
\in\wih Y\setminus\wih Z.
\end{equation}
So by \eqref{e:wiFSpot},\ke\eqref{e:psipk},\ke\eqref{e:peak2} 
we get that 
\[\wih\varphi_p^{\wih Y}=\wih\varphi|_{\wih Y}+\frac{1}{2p}\,
\log\wih P^{\wih Y}_p\geq
\frac{1}{p}\,\left((p-p_0)\psi_k+p_0\rho+
\sum_{j=1}^\ell\log\sigma_j-\frac{\log C_2}{2}-p_0a\right) 
\,\text{on } \wih Y\setminus\wih Z,\]
for $p\geq p_0$ and $k\geq1$. Letting $k\to\infty$
and using \ke\eqref{e:psik},\ke\eqref{e:eqreq},\ke\eqref{e:rho} we obtain 
\[\wih\varphi_p^{\wih Y}\geq\big(\wih\varphi^{\wih Y}_\eq\big)^\star+
\frac{1}{p}\,\left(p_0\eta-
p_0\big(\wih\varphi^{\wih Y}_\req\big)^\star+
\sum_{j=1}^\ell\log\sigma_j-\frac{\log C_2}{2}-
p_0a\right)\,\text{on } \wih Y\setminus\wih Z,\]
for $p\geq p_0$. Here $\wih\varphi^{\wih Y}_\req$ is the 
reduced equilibrium envelope of $(\wih\alpha,\wih Y,
\wih\Sigma,\tau,\wih\varphi)$ defined in \eqref{e:envreq}. 
The conclusion follows since 
$\big(\wih\varphi^{\wih Y}_\req\big)^\star$ is 
bounded above on $\wih Y$ and $\eta\leq-1$.
\end{proof}


\begin{proof}[Proof of Theorem \ref{T:wiFSpot}] By the \L ojasiewicz inequality, 
there exist constants $N_j,M_j>0$, $1\leq j\leq\ell$, such that 
$\log\sigma_j(x)\geq-N_j\big|\log\dist\big(x,\wih\Sigma_j\big)\big|-M_j$, $x\in\wih X$. 
Using Proposition \ref{P:lewiFS} and \eqref{e:eta2} we infer that there exist $C_1>0,\,p_0\in\N$ such that  
\begin{equation}\label{eq:lowbound}
\wih\varphi^{\wih Y}_p(y)\geq \big(\wih\varphi^{\wih Y}_\eq\big)^\star(y)-
\frac{C_1}{p}\,\Big(\big|\log\dist\big(y,\wih Z\big)\big|+1\Big),\;y\in \wih Y,\;p\geq p_0.
\end{equation}
Note that $\log\dist\big(\LargerCdot,\wih Z\big)|_{\wih Y}\in 
L^1(\wih Y,\wih\omega^m|_{\wih Y})$, 
since $\wih Y\not\subset\wih Z$ (see e.\ke g.\ \cite[Lemma 5.2]{CMN16} and its proof). 
The proof of Theorem \ref{T:wiFSpot} now proceeds exactly as that of \cite[Theorem 5.1]{CMN19}
by using the lower bound \eqref{eq:lowbound} 
and the upper bound from Proposition \ref{P:uewiFS}. 
\end{proof}

\begin{proof}[Proof of Theorem \ref{T:FSpot}] Since $\big(\wih\varphi^{\wih Y}_\eq\big)^\star$ is 
$\wih\alpha|_{\wih Y}$-psh, we have $\big(\wih\varphi^{\wih Y}_\eq\big)^\star\leq M$ on $\wih Y$ for some constant $M$. Recall from Lemma \ref{L:embdes2} that
$\wih\pi:\wih X\setminus(\wih E_\pi\cup\wih E)\to X\setminus Z$ is a 
biholomorphism. Therefore the function
\begin{equation}\label{e:}
\varphi^Y_\eq:=\big(\wih\varphi^{\wih Y}_\eq\big)^\star\circ\wih\pi^{-1}
\end{equation}
is $\alpha|_Y$-psh and $\varphi^Y_\eq\leq M$ on $Y_\reg\setminus Z$, 
hence it extends to a $\alpha|_Y$-psh function on $Y_\reg$ which is bounded above by $M$. 
This shows that $\varphi^Y_\eq$ is a weakly $\alpha|_Y$-psh function on $Y$. 
Moreover, by \eqref{e:wiFSpot} and since 
$\wih\omega\geq\wih\pi^\star\omega$ we have 
\[\int_{Y\setminus Z}\big|\varphi^Y_p-\varphi^Y_\eq\big|\,\omega^m=
\int_{\wih Y\setminus(\wih E_\pi\cup\wih E)}\big|\wih\varphi^{\wih Y}_p-
\big(\wih\varphi^{\wih Y}_\eq\big)^\star\big|\,\wih\pi^\star\omega^m\leq
\int_{\wih Y}\big|\wih\varphi^{\wih Y}_p-
\big(\wih\varphi^{\wih Y}_\eq\big)^\star\big|\,\wih\omega^m.\]
Theorem \ref{T:FSpot} now follows from Theorem \ref{T:wiFSpot}.
\end{proof}

When $\varphi$ is smooth, we may obtain a  more precise   estimate for $\frac{\wih P^{\wih Y}_p}{p^m}$, and hence for  $\wih \varphi^{\wih Y}_p$ as in  \cite{CM17}. 

 
\section{Zeros of random holomorphic sections}\label{S:zrhs}

We deal here with the proof of Theorem \ref{T:zrhs}. It is very similar to the proof of \cite[Theorem 1.10]{CMN19}, so we will only give an outline. The first step is to show that zero divisors of random sections distribute like the Fubini-Study currents.

\begin{Theorem}\label{T:speed}
Let $X,Y,L,\Sigma,\tau$ verify assumptions (A)-(E), let $h$ be a bounded singular Hermitian 
metric on $L$, and assume that $(L,\Sigma,\tau)$ is big and there exists a K\"ahler form 
$\omega$ on $X$. Then there exists a constant $c>0$ with the following property: 
For any sequence of positive numbers $\{\lambda_p\}_{p\geq1}$ such that 
\[\liminf_{p\to\infty} \frac{\lambda_p}{\log p}>(1+m)c\,,\]
there exist subsets $E_p\subset\X^Y_p$ such that 

(a) $\sigma_p(E_{p})\leq cp^m\exp(-\lambda_p/c)$ holds for all $p$ sufficiently large;

(b) if $s_p\in\X^Y_p\setminus E_p$ we have 
\[\Big|\frac{1}{p}\,\langle[s_p=0]-\gamma^Y_p,\phi\rangle\Big|\leq
\frac{c\lambda_p}{p}\,\| \phi\|_{\cC^2}\,,\] 
for any $(m-1,m-1)$-form $\phi$ of class $\cC^2$ on $Y$. 

In particular, the last estimate holds for 
$\sigma_\infty$-{\ke}a.{\ke}e.\ $\{s_p\}_{p\geq1}\in\X^Y_\infty$ 
provided that $p$ is large enough.
\end{Theorem} 

\begin{proof} We follow closely the proof of \cite[Theorem 6.1]{CMN19} and apply  
the Dinh-Sibony equidistribution theorem for meromorphic transforms \cite[Theorem 4.1]{DS06}.
Recall by \eqref{e:iso} that the spaces $H^0_{0,(2)}(\wih X|\wih Y,\wih L^p)),H^0_{0,(2)}(X|Y,L^p)$ are isometric. Using the notation from Section \ref{SS:FSpot}, we first show that 
Theorem \ref{T:speed} holds on $\wih Y$ for the spaces 
\[\wih\X^{\wih Y}_p:=\P H^0_{0,(2)}(\wih X|\wih Y,\wih L^p),\;\sigma_p=
\omega_\FS^{d_p},\;(\wih\X^{\wih Y}_\infty,\sigma_\infty):= 
\prod_{p=1}^\infty (\wih\X^{\wih Y}_p,\sigma_p),\;d_p=\dim\wih\X^{\wih Y}_p=\dim\X^Y_p,\] 
and the Fubini-Study currents $\wih\gamma^{\wih Y}_p$.
This is done exactly as in the proof of Theorem 6.1, Step 1, from \cite{CMN19} (see also \cite[Section 4]{CMN16}), by applying \cite[Theorem 4.1]{DS06} to the Kodaira maps considered as meromorphic transforms of codimension $m-1$, 
$\Phi_p:\wih Y\dashrightarrow\P H^0_{0,(2)}(\wih X|\wih Y,\wih L^p)$, with graph 
\[\Gamma_p=\big\{(y,\hat s)\in\wih Y\times\P H^0_{0,(2)}(\wih X|\wih Y,\wih L^p):\,\hat s(y)=0\big\}.\]
Note that Siegel's lemma implies that $d_p=O(p^m)$.

We next show that Theorem \ref{T:speed} holds on $Y$ for the spaces $\X^Y_p$. Consider the restriction $\wih\pi:=\wih\pi|_{\wih Y}:\wih Y\to Y$. By Lemma \ref{L:embdes2} $\wih\pi:\wih Y\setminus(\wih E_\pi\cup\wih E)\to Y\setminus Z$ is a biholomorphism, and by \eqref{e:iso} $S\in H^0_{0,(2)}(X|Y,L^p)\to\wih\pi^\star S\in H^0_{0,(2)}(\wih X|\wih Y,\wih L^p)$ is an isometry. Using \eqref{e:wiFS} and \eqref{e:wiFSpot} we obtain  
\[\frac{1}{p}\,\wih\gamma^{\wih Y}_p=
\wih\alpha|_{\wih Y}+dd^c\wih\varphi^{\wih Y}_p=
\wih\pi^\star\alpha+dd^c(\varphi^Y_p\circ\wih\pi).\]
Since $\varphi^Y_p\in L^1(Y,\omega^m|_Y)$ and 
$\wih\varphi^{\wih Y}_p=\varphi^Y_p\circ\wih\pi\in L^1(\wih Y,\wih\omega^m|_{\wih Y})$
we infer that $\wih\pi_\star\wih\gamma^{\wih Y}_p=\gamma^Y_p$ as currents on $Y$. 
Similarly we can show that $\wih\pi_\star[\wih\pi^\star S=0]=[S=0]$ as currents on $Y$, 
for $S\in H^0_{0,(2)}(X|Y,L^p)$. Theorem \ref{T:speed} now follows from the above 
considerations by arguing as in the proof of Theorem 6.1, Step 2, from \cite{CMN19}.
\end{proof}

\begin{proof}[Proof of Theorem \ref{T:zrhs}] 
Theorem \ref{T:zrhs} follows easily from Theorems \ref{T:FSpot} and \ref{T:speed}, by 
proceeding as in the proof of \cite[Theorem 1.10]{CMN19}.
\end{proof}

\begin{Remark}\label{R:A}
Assume that $X,L,\Sigma,\tau$ verify (A)-(D), $X$ is smooth and $\dim\Sigma_j=n-1$, 
$1\leq j\leq\ell$. Then Theorems \ref{T:FSpot} and \ref{T:zrhs} hold for any analytic 
subset $Y\subset X$ that verifies assumption (E*) from Proposition \ref{P:A}.
\end{Remark}


\end{document}